\documentclass[11pt]{article}   
\usepackage{amsfonts,amsmath,latexsym}
\usepackage[T1]{fontenc}
\usepackage[utf8]{inputenc}
\usepackage{algorithm}
\usepackage{algorithmic}
\usepackage{eurosym}
\usepackage{enumerate}
\usepackage{dsfont}
\usepackage{showkeys}
\usepackage{hyperref}
\usepackage[dvips]{graphicx}
\usepackage{epsfig}
\usepackage{enumerate}
\usepackage{epsf}   
\usepackage{amsthm}
\usepackage{color}
\usepackage{amssymb}

\usepackage{comment}

\usepackage{caption}
\usepackage{subfigure}

\usepackage{fancyhdr} 

\usepackage{palatino}

\newtheorem{thm}{Theorem}[section]
\newtheorem{prop}[thm]{Proposition}
\newtheorem{lem}[thm]{Lemma}
\newtheorem{corro}[thm]{Corollary}

\newtheorem{rem}[thm]{Remark}

\newtheorem{ass}{Assumption}

\newenvironment{prooff}[1]{\begin{trivlist}
\item {\it \bf Proof}\quad} {\qed\end{trivlist}}

\def\R{\mathbb R}
\def\N{\mathbb N}

\def\E{\mathbb E}

\def\Q{\mathbb Q}
\def\sha{{\cal A}}

\def\shc{{\cal C}}

\def\shf{{\cal F}}

\def\shj{{\cal J}}

\def\shs{{\cal S}}

\author{
{\sc Lucas IZYDORCZYK}
\thanks{ENSTA Paris, Institut Polytechnique de Paris.
Unit\'e de Math\'ematiques Appliqu\'ees (UMA).
 E-mail:{ \tt lucas.izydorczyk@ensta-paris.fr}} 
{\sc,}\ {\sc Nadia OUDJANE}
\thanks{EDF R\&D,   and FiME (Laboratoire de Finance des March\'es de l'Energie
(Dauphine, CREST,  EDF R\&D) www.fime-lab.org). 
E-mail:{\tt  
nadia.oudjane@edf.fr}}
\ {\sc and}\ {\sc Francesco RUSSO} 
\thanks{ENSTA Paris, Institut Polytechnique de Paris.
Unit\'e de Math\'ematiques Appliqu\'ees (UMA). 
 E-mail:{\tt  francesco.russo@ensta-paris.fr}.
 }}

\date{September 10th 2021}

\title{A fully backward representation of semilinear PDEs 
applied to the control of thermostatic
loads in power systems }

\newcommand{\MBFigure}[6]{
$\left. \right.$ \\
\refstepcounter{figure}
\addcontentsline{lof}{figure}{\numberline{\thefigure}{\ignorespaces #5}}
\begin{center}
\begin{minipage}{#1cm}
\centerline{\includegraphics[width=#2cm,angle=#3]{#4}}
\begin{center}
\upshape{F\textsc{ig} \normal
\end{center}
size{\thefigure}. $-$} #5
\end{center}
\label{#6}
\end{minipage}
\end{center}
$\left. \right.$ \\}

\begin{document}
\maketitle
 \begin{abstract}
We propose a fully backward representation
of semilinear PDEs with application to stochastic control.
Based on this, we  develop a fully backward Monte-Carlo  scheme allowing to
generate the regression grid, backwardly in time, as the value function is
computed. This offers two key advantages in terms of computational
efficiency and memory. First, the grid is generated adaptively in the areas of interest and second, there is no need to store the entire grid. 
The performances of this technique are compared in simulations 
to the traditional Monte-Carlo forward-backward approach
on a control problem of thermostatic loads.
 \end{abstract}

\medskip\noindent  {\bf 2020  AMS-classification}: 60H10; 
 60H30; 60J60; 65C05; 49L25;
35K58.
 
\section{Introduction}



The numerical resolution of non-linear PDEs is a crucial issue
in many applications. In particular, stochastic control problems
can be formulated by mean of the Hamilton-Jacobi-Bellman (HJB) equations
with terminal condition.
In this paper, we focus more particularly on control problems
raised by  demand-side management in power systems.
The difficulties come especially from the high dimensionality of the state space,
which motivates the use of probabilistic representations.
The main issue of numerical schemes is then to concentrate
the computing effort in specific regions of interest in the state space.
In classical regression Monte-Carlo approaches, the solution is evaluated
backwardly in time from the final time to the initial time, while
the regression grid is generated forwardly from
the initial time to the final one.
In this paper, we propose a fully backward probabilistic approach
which allows to generate adaptively the regression grid, as the solution is evaluated, taking advantage of the calculations already performed. Besides, there is no need to store the entire grid, since the points are generated as they are used for calculations. 
Our grid will be indeed simulated according to the time-reversal
of some diffusion starting from a judicious terminal distribution.

\noindent
We are interested in semilinear PDEs of the type
\begin{equation}
\label{eq:PDE_Intro}
\left\{
\begin{array}{lll}
\partial_t v (t,x)+H(t,x, v(t,x),\nabla_x v(t,x)) +\frac{1}{2} Tr[\sigma \sigma^\top (t) \nabla^2_x v(t,x)]=0, \ \left(t,x\right) \in [0,T[\times\R^d\\ 
v(T,x)=g(x),
\end{array}
\right .
\end{equation}
where in particular $\sigma$ is a deterministic non-degenerate
 matrix-valued function.
Under suitable conditions, there exists a unique
viscosity solution $v$ of \eqref{eq:PDE_Intro} in
the class of continuous functions with polynomial growth.
One classical probabilistic representation of $v$ is provided 
by Forward-Backward SDEs (FBSDEs), see e.g. \cite{pardouxgeilo}.
First a forward diffusion is fixed, with an arbitrary drift $\tilde b$ 
\begin{equation}\label{ForwE}
 dX_t = \widetilde{b}(t,X_t) dt + \sigma(t) dW_t.
\end{equation}
Then the solution of \eqref{eq:PDE_Intro} is represented by 
$v(s,x) = Y^{s,x}_s$,
where $ \left(Y,Z\right) = \left(Y^{s,x},Z^{s,x}\right)$
is the unique solution of the BSDE
\begin{equation} \label{BSDEIntro}
 Y_t = g(X_T) + \int_t^T F(r,X_r, Y_r,Z_r) dr - \int_t^T Z_r dW_r,
\end{equation}
with $X = X^{s,x}$ being the solution of \eqref{ForwE}
starting at time $s$ with value $x$ and 
$F$ being related to $H$ by  
\begin{equation} \label{Eq_F}
  F(t,x,y,z) := H(t,x,y,\left(\sigma^{-1}\left(t\right)\right)^\top z) - \left<\widetilde{b}(t,x),\left(\sigma^{-1}\left(t\right)\right)^\top z\right>.
  \end{equation}
Considering a time discretization mesh $t_{k}=k\delta t$, with $\delta t=\frac{T}{n}$ and $k=0,\cdots, n$, for a given positive integer $n$, \cite{labart}
 proved that one can approximate $(Y_{t_k}, Z_{t_k})$ by $(\hat Y_k, \hat Z_k)$ such that $\hat Y_n=g(X_T)$ and for $k=0,\cdots, n-1$ 
\begin{equation} \label{CondExp}
\left\{
\begin{array}{lll}
\hat Y_{k} &=& \displaystyle {\mathbb{E}\left (\sum_{\ell =k+1}^n
               F(t_\ell,X_{t_\ell},\hat Y_{\ell},\hat Z_{\ell-1})\delta t + g(X_T )\middle| X_{t_k}\right)
               }\\
\hat Z_{k} &=& \displaystyle {\frac{1}{\delta t}\mathbb{E}\left (\hat Y_{{k+1}}(W_{t_{k+1}}-W_{t_k}) \middle| X_{t_k}\right )}.
\end{array}
\right.
\end{equation}
Most of probabilistic numerical schemes (see e.g. regression Monte-Carlo~\cite{GobetWarin,bender12}, Kernel Monte-Carlo~\cite{BouchardTouzi}, Quantization~\cite{DelarueMenozzi})
rely on that representation. 
The common idea is then articulated in two steps. First,
one generates a grid discretizing the forward process~\eqref{ForwE} in space and time on $[0,T]$, (by Monte-Carlo simulations or Quantization, etc.). Then,
one calculates the conditional expectations~\eqref{CondExp} on the grid points in order to estimate $(\hat Y,\hat Z)$.
These techniques have generally two limitations.
\begin{enumerate}
\item The degree of freedom in the choice
  of the forward diffusion $X$ is difficult to exploit although it has a major impact on the numerical scheme efficiency: how to chose a reasonable drift $\tilde b$ without a priori information on $v$ ?
\item 
 The entire grid discretizing the forward process has to be stored in memory to be  revisited backwardly in time in order to compute the
solution process $(Y,Z)$. This approach naturally raises some huge memory issues 
which in general limit drastically the number of Monte-Carlo runs and time steps, hence 
the accuracy of the procedure.
\end{enumerate}
To overcome such limitations some approaches were proposed 
in the domain of mathematical finance, in particular for the evaluation of American style options. One technique, intended to deal with the memory problem, relies on bridge simulation,
see e.g. \cite{ribeiro03, sabino20}.
However, 
this approach requires specific developments for each price model (based for instance on the Brownian 
bridge for Brownian prices or on the gamma bridge for variance gamma prices) and remains difficult to
generalize to a wide class of models. To address the efficiency issue, \cite{bender07} developed a  scheme based on  Picard's type iterations that avoids the use of nested conditional expectations backwardly in time, which are replaced by nested conditional expectations along the iterations. 
In the same line, \cite{gobet10bis} proposes an adaptive variance reduction technique which combines Picard’s iterations and control variate to solve the BSDE. A parallel version of that algorithm was proposed in~\cite{labart13}. However, those approaches require, at each iteration, to approximate the solution on the whole time horizon. Similarly, importance sampling and Girsanov's theorem, were considered to force the exploration of the space towards areas of interest~\cite{bender10}. In particular, this type of approach was derived in the case of stochastic control in~\cite{exarchos} providing an iterative scheme that is capable of
learning the optimally controlled drift. Here again, that method requires several
estimations of the value function on the whole time horizon.
Besides~\cite{gobet17} proposed an adaptive importance sampling scheme for FBSDEs allowing to select the drift adaptively, as the calculations are performed backwardly. Unfortunately, that approach is limited to situations where the driver
$F$ does not depend on $Z$. In the present paper, we introduce a new adaptive
approach to address both the memory problem and the efficiency issue (related to the drift selection) in the general case where the driver may depend on $X,Y$ and $Z$. 

\noindent
We propose to choose
adaptively the drift $\tilde b$ 
 at the same time as we discover the function
$v$ such that
\begin{equation} \label{RepFormulaIntro}
\displaystyle v\left(t,X_t\right) = \mathbb{E}\left (\int^{T}_{t}
H\left(s,X_s,v\left(s,X_s\right),
\nabla_xv\left(s,X_s\right)\right) - \left<\widetilde{b}\left(s,X_s\right),\nabla_x v\left(s,X_s\right)\right>ds + g\left(X_T \right)\middle| X_t\right),
\end{equation}
by simulating the time-reversal
of a solution $X$ of \eqref{ForwE} starting from the distribution of $X_T$.
More specifically, to take advantage of the Ornstein-Uhlenbeck setting, we choose the drift $\tilde b$ to be affine w.r.t. the space variable. 
We  fix a Gaussian distribution $\nu$
and look for solutions $\xi$ of the McKean SDE
\begin{equation} \label{Rev-SDEIntro}
\begin{cases} 
\xi_0 \sim \nu, \\
\displaystyle \xi_t = \xi_0 - \int^{t}_{0}\widetilde{b}\left(T-s,{\xi}_s\right) + \sigma\sigma^\top\left(T-s\right)Q\left(T-s\right)^{-1}\left(\xi_s - m\left(T-s\right)\right)ds + \int^{t}_{0}\sigma\left(T-s\right)d\beta_s, \\
m(T- t) = \E\left(\xi_t\right), \\
Q(T- t) = {\rm Cov}\left(\xi_t\right)\quad \textrm{for}\  t \in ]0,T] .
\end{cases}
\end{equation}
By Proposition \ref{McKean}, \eqref{Rev-SDEIntro}
admits exactly one solution $\xi$, provided 
Assumption \ref{ass_nu} in Section \ref{R41}
is verified. That
 assumption depends on the covariance matrix of $\nu$, the drift $\tilde b$ and the volatility $\sigma$. Indeed, one important limitation is that the covariance matrix should be chosen carefully to ensure that the process is well-defined until $T$.   
Point 2. of Proposition \ref{McKean} and Lemma \ref{HP_lemma}
say that the time-reversal process $\hat{\xi}$, i.e. $\hat \xi_t:=\xi_{T-t}$, 
is an Ornstein-Uhlenbeck process solution 
of \eqref{ForwE} such that the law of $X_0$ is
Gaussian with mean $m(0)$ and covariance $Q(0)$. 
%
This leads to the first result of this paper which consists of the fully backward representation
 stated in Theorem~\ref{SLPDE-Rep}.
The proof is based on Feynman-Kac type formula instead of BSDEs and it does
not require explicitly the uniqueness of viscosity solution of the PDE \eqref{eq:PDE_Intro}.
The second contribution of the paper is Corollary \ref{RepFormulaControl}
which is the ``instantiation'' of  Theorem~\ref{SLPDE-Rep}
 in the framework of
stochastic control, i.e. the representation of its value function (solution of a  Hamilton-Jacobi-Bellman equation). This holds 
when the running and terminal cost have polynomial growth with respect to the state space variable.
We also suppose that the value function is of class $\shc^{0,1}$ whose gradient has polynomial growth.
In particular, we derive in Corollary \ref{HJB_corro}, a representation  
involving 
the gap between the optimally controlled drift and the instrumental drift $\tilde b$.
In Section \ref{S5}, we present a fully backward Monte-Carlo regression scheme,
where the instrumental drift  is adaptively updated in order to mimic the optimally 
controlled dynamics, see Algorithm \ref{algo}.
We expect that this approach is particularly well-suited
when the final cost has a strong impact on the global cost
and  when the terminal cost function
 is  localized in a small region of the space, so that the initial distribution $\nu$
 can be chosen in an appropriate way. 
Finally, in Section \ref{Sexample} we illustrate the interest of this new algorithm
applied to the problem of controlling the consumption of a large number of thermostatic loads in order to
minimize an aggregative cost. 
We compare our approach to the classical regression Monte-Carlo scheme based on a forward grid.

\section{Notations}

\label{SNotat}

\noindent Let us fix $T > 0$, $d,k \in \mathbb{N}^*$. For a given $p \in \N^*$, $[\![1,p]\!]$ denotes the set of all integers between $1$ and $p$ included. $\left<\cdot,\cdot\right>$ denotes the usual scalar product on $\R^d$ and $\left|\cdot\right|$ the associated norm. Elements of $\R^d$ are supposed to be column vectors. $M_d\left(\R\right)$ stands for the set of $d \times d$ matrices, $S_d\left(\R\right)$ for the subset of symmetric matrices, $S^+_d\left(\R\right)$ the subset of symmetric positive semi-definite
matrices (in particular  with non-negative eigenvalues)
 and $S^{++}_d\left(\R\right)$ for the 
subset of strictly positive definite
symmetric matrices.
For a given $A \in M_d\left(\R\right)$, $A^\top$ will denote its transpose, $Tr\left(A\right)$ its trace, $Sp\left(A\right)$ its spectrum,
i.e. the set of its eigenvalues,
$e^A := \sum^{\infty}_{k=0}\frac{A^k}{k!}$ its exponential and $\left|\left|A\right|\right| := \sup_{x\in\R^d, \left|x\right| = 1} \left|Ax\right|$. For a given $A \in S^{+}_d\left(\R\right)$, $\sqrt{A}$ denotes the unique element of $S^{+}_d\left(\R\right)$ such that $(\sqrt{A})^2 = A$.

\smallbreak
\noindent
For a given continuous function $f : [0,T] \mapsto \R^d$ (resp. $g : [0,T] \mapsto M_d\left(\R\right)$), we set $\left|\left| f\right|\right|_{\infty} := \sup_{t \in [0,T]} \left| f\left(t\right)\right|$ (resp. $\left|\left| g \right|\right|_{\infty} := \sup_{t \in [0,T]} \left|\left| g\left(t\right)\right|\right|$). 
 $\shc^{1,2}\left([0,T],\R^d\right)$ (resp. $\shc^{0,1}\left([0,T],\R^d\right)$) denotes the set of real-valued functions defined on  $[0,T]\times\R^d$
 being continuously differentiable in time and twice continuously differentiable in space (resp. continuous in time and continuously differentiable in space). $\shc^0\left([0,T]\times\R^d\right)$ (resp $\shc^1\left(\R^d\right)$) denotes the set of continuous (resp  continuously differentiable) real-valued functions defined on  $[0,T]\times\R^d$ (resp. $\R^d$). $\nabla_x$ will denote the gradient operator and $\nabla^2_x$ the Hessian matrix. For each $p \in \N$, $P_p\left(\R^d\right)$ denotes the set of polynomial functions on $\R^d$ with degree $p$. 
\smallbreak
\noindent In the whole paper, we say that a function $v : [0,T]\times\R^d \mapsto \R$ has \textit{polynomial growth} if there exists $q,K > 0$ such that for all $\left(t,x\right) \in [0,T]\times\R^d$
\begin{equation*}
\left|v\left(t,x\right)\right| \leq K\left(1 + \left|x\right|^q\right).
\end{equation*}
\smallbreak
\noindent When $v$ verifies previous property with $q = 1$, we say that it has \textit{linear growth}.
\smallbreak
\noindent For a given random vector $X$ defined on a probability space
$\left(\Omega,\mathcal{F},\mathbb{P}\right)$,
$\mathbb{E}_{\mathbb{P}}\left(X\right)$ (resp. $\mathrm{Cov}_{\mathbb{P}}\left(X\right) := \mathbb{E}_{\mathbb{P}}\left(\left(X - \mathbb{E}_{\mathbb{P}}\left(X\right)\right)\left(X - \mathbb{E}_\mathbb{P}\left(X\right)\right)^\top\right)$)
will denote
its expectation (resp. its covariance matrix) under $\mathbb{P}$. 
 When self-explanatory, the subscript will be omitted in the sequel. For a given $\left(m,Q\right) \in \R^d\times S^{+}_d\left(\R\right)$, $\mathcal{N}\left(m,Q\right)$ denotes the Gaussian probability on $\R^d$ with mean $m$ and covariance matrix $Q$.
\smallbreak
\noindent For any stochastic process $X$, $\mathcal{F}^{X}$ will denote its canonical filtration. $\widehat{X}$ will denote the time-reversal process
$X_{T-\cdot}$.

\section{Representation of semilinear PDEs} \label{RepSection}

\setcounter{equation}{0}
\subsection{Around two backward ODEs} 

\label{R41}
\noindent Let $a$ (resp. $c$) be Borel bounded functions
from $[0,T]$ to $M_d\left(\R\right)$ (resp. $\R^d$).
\smallbreak
\noindent
In the sequel we will fix 
a Gaussian Borel probability $\nu$ on $\R^d$ 
 with mean $\bar{m}^\nu$ and covariance matrix $\bar{Q}^\nu$.
We consider the functions $m^\nu : [0,T] \mapsto \R^d$ and $Q^\nu : [0,T] \mapsto S_d\left(\R\right)$ denoting respectively the unique solutions of the backward ODEs
\begin{equation}
\label{ODE_m}
\begin{cases}
\frac{d}{dt}m^\nu\left(t\right) = a\left(t\right)m^\nu\left(t\right) + c\left(t\right), \quad  t \in [0,T]\\
m^\nu\left(T\right) = \bar{m}^\nu,
\end{cases}
\end{equation}
\begin{equation}
\label{ODE_Q}
\begin{cases}
\frac{d}{dt}Q^\nu\left(t\right) = Q^\nu\left(t\right)a\left(t\right)^\top + a\left(t\right)Q^\nu\left(t\right) + \Sigma\left(t\right), t \in [0,T] \\
Q^\nu\left(T\right) = \bar{Q}^\nu,
\end{cases}
\end{equation}
\noindent for which existence and uniqueness hold since they are linear.
\smallbreak
\noindent We introduce an hypothesis on $\nu$ which will be used in
the sequel.
\begin{ass} \label{ass_nu}
$Q^\nu(0) \in \shs^+_d\left(\R\right)$.
\end{ass}
\noindent Easy computations imply for all $t \in [0,T]$
\begin{equation} \label{m_explicit}
m^\nu\left(t\right) = \sha\left(t\right)\left(\sha\left(T\right)^{-1}\bar{m}^\nu - \int^{T}_{t}\sha\left(s\right)^{-1}c\left(s\right)ds\right),
\end{equation}
\begin{equation} \label{Q_explicit}
Q^\nu\left(t\right) = \sha\left(t\right)\left(\sha\left(T\right)^{-1}\bar{Q}^\nu\left(\sha\left(T\right)^{-1}\right)^\top - \int^{T}_{t}\sha\left(s\right)^{-1}\Sigma\left(s\right)
\left(\sha\left(s\right)^{-1}\right)^\top ds \right)\sha\left(t\right)^\top,
\end{equation}
\noindent where $\sha\left(t\right),t\in[0,T]$ is the unique solution of the matrix ODE
\begin{equation}
\label{ODE_funda}
\begin{cases}
\frac{d}{dt}\sha\left(t\right) = a\left(t\right)\sha\left(t\right), t \in [0,T]\\
\sha\left(0\right) = I_d.
\end{cases}
\end{equation}
\noindent We recall that for all $t \in [0,T]$, $\sha\left(t\right)$ is invertible and the matrix valued function $t \mapsto \sha(t)^{-1}$ solves the ODE
\begin{equation}
\label{ODE_funda_inv}
\begin{cases}
\frac{d}{dt}\sha\left(t\right)^{-1} = -\sha\left(t\right)^{-1}a\left(t\right), t \in [0,T]\\
\sha\left(0\right)^{-1} = I_d,
\end{cases}
\end{equation}
\noindent see Chapter 8 in \cite{bronson} for similar and further properties.
\smallbreak
\noindent Note that in the case $a\left(t\right) = a, \  t \in [0,T]$ for a given $a \in M_d\left(\R\right)$, then  $\sha : t \rightarrow e^{at}$ and identities \eqref{m_explicit}, \eqref{Q_explicit} simplify as follows:
\begin{equation} \label{m_explicit_simple}
m^\nu\left(t\right) = e^{-a\left(T-t\right)}\bar{m}^\nu - \int^{T}_{t}e^{-a\left(s-t\right)}c\left(s\right)ds,
\end{equation}
\begin{equation} \label{Q_explicit_simple}
Q^\nu\left(t\right) = e^{-a\left(T-t\right)} \bar{Q}^\nu e^{-a^\top\left(T-t\right)} - \int^{T}_{t}e^{-a\left(s-t\right)}\Sigma\left(s\right)e^{-a^\top\left(s-t\right)}ds,
\end{equation}
\noindent for all $t\in [0,T]$.

\begin{rem}\label{rem_nondegen}
\noindent Suppose  that $Q^\nu\left(0\right)$ belongs to $S^+_d\left(\R\right)$. Identity \eqref{Q_explicit} gives in particular 
\begin{equation} \label{Q_explicit_fwd}
Q^\nu\left(t\right) = \sha\left(t\right)\left(Q^\nu\left(0\right) + \int^{t}_{0}\sha\left(s\right)^{-1}\Sigma\left(s\right)
\left(\sha\left(s\right)^{-1}\right)^\top ds \right)\sha\left(t\right)^\top, \ t \in [0,T].
\end{equation}
\noindent Combining \eqref{Q_explicit_fwd} and the fact $\sigma\left(t\right)$ is invertible for all $t \in [0,T]$, we remark that $Q^\nu\left(t\right)$ belongs to $S^{++}_d\left(\R\right)$ for all $t \in ]0,T]$.
\end{rem}
\noindent Finally we give a condition depending on $\sha, \sigma, \bar{Q}^\nu$ and $T$ to ensure the measure $\nu$ fulfills Assumption \ref{ass_nu}.
\begin{prop} \label{Q_pos_cond}
  \noindent Suppose that 
\begin{equation} \label{cond_pos}
\min Sp\left(\bar{Q}^\nu\right) \geq \int^{T}_{0}\left|\left|\sigma\left(s\right)\right|\right|^2\left|\left|\left(\sha\left(T\right)\sha\left(s\right)^{-1}\right)^\top\right|\right|^2ds.
\end{equation}
\noindent Then,
\begin{equation} \label{P42Concl}
  Q^\nu\left(0\right) \in S^+_d\left(\R\right).
  \end{equation}
\end{prop}
\begin{proof}
  \noindent
 Since $\sha\left(T\right)$ is invertible and
  ${Q}^\nu\left(0\right)$
  belongs to $S_d\left(\R\right)$, 
  \eqref{P42Concl}  is equivalent to

\begin{equation}\label{EP42}
 \sha\left(T\right)Q^\nu\left(0\right)\sha\left(T\right)^\top
  \in
  S^+_d\left(\R\right).
\end{equation}
 To prove \eqref{EP42}, 
taking into account
\eqref{Q_explicit},
it suffices to show that the matrix 
\begin{equation*}
\bar{Q}^\nu - \int^{T}_{0}\sha\left(T\right)\sha\left(s\right)^{-1}\Sigma\left(s\right)
\left(\sha\left(T\right)\sha\left(s\right)^{-1}\right)^\top ds \in S^+_d\left(\R\right),
\end{equation*}
\noindent or, equivalently, that for all $x \in \R^d$
\begin{equation}\label{TxQx}
\lambda := x^\top \bar{Q}^\nu x - \int^{T}_{0}x^\top\sha\left(T\right)\sha\left(s\right)^{-1}\Sigma\left(s\right)
\left(\sha\left(T\right)\sha\left(s\right)^{-1}\right)^\top x ds \geq 0.
\end{equation}
\noindent
Let $x \in \R^d$, 
\begin{align*}
\lambda &{} \geq \min Sp\left(\bar{Q}^\nu\right)\left|x\right|^2 - \int^{T}_{0}\left|\sigma\left(s\right)^\top
\left(\sha\left(T\right)\sha\left(s\right)^{-1}\right)^\top x\right|^2ds, \\
&{} \geq \left(\min Sp\left(\bar{Q}^\nu\right) - \int^{T}_{0}\left|\left|\sigma\left(s\right)^\top\right|\right|^2\left|\left|
\left(\sha\left(T\right)\sha\left(s\right)^{-1}\right)^\top\right|\right|^2ds\right)\left|x\right|^2, \\
&{} \geq 0, 
\end{align*}
\noindent since \eqref{cond_pos} holds. This ends the proof.
\end{proof}
\begin{rem}
\smallbreak
\noindent In the case $a\left(t\right) = a, \  t \in [0,T]$ for a given $a \in M_d\left(\R\right)$, Condition \eqref{cond_pos} is satisfied in particular if 
\begin{equation} \label{cond_pos_simple}
\min Sp\left(\bar{Q}^\nu\right) \geq  \left|\left|\sigma^\top\right|\right|^2_{\infty}\int^{T}_{0}\left|\left|e^{a^\top s}\right|\right|^2ds
\end{equation}
is verified.
\end{rem}


\begin{rem} \label{RmCov}
Let $X$ be a solution of 
\begin{equation} \label{EOUIntro}
  X_t = X_0 + \int^{t}_{0} \widetilde{b}\left(s,X_s\right) ds
  + \int^{t}_{0}\sigma\left(s\right)dW_s, \ t \in [0,T[, 
\end{equation}
where $\sigma$ is a deterministic matrix-valued function and
$\widetilde{b}$ the piecewise affine function
$$ \widetilde{b} (t,x) = a(t) x + c(t), \ t \in [0,T],$$
 and $X_0$ be a square integrable r.v.
It is well-known that $X$ is a square integrable process.
 Let, for every $t \in [0,T]$,
 $m(t)=\mathbb{E}\left(X_t\right)$ and $Q(t)$ the covariance matrix of $X_t$. 
Setting $\bar m^\nu= \mathbb{E}\left(X_T\right)$ and $\bar Q^\nu$ the covariance matrix
of $X_T$. Then 
\begin{equation} \label{EEnu}
m = m^\nu, \quad Q= Q^\nu.
\end{equation}
Indeed, by Problem 6.1 in Chapter 5 in \cite{karatshreve}
  $m$ (resp. $Q$) 
is solution of  \eqref{ODE_m} (resp. \eqref{ODE_Q}).
\eqref{EEnu} follows by uniqueness of previous ODEs.
\end{rem}

\subsection{The representation formula for a general semilinear PDE}

\label{S42}
\noindent In the whole paper $\sigma$ will be a continuous function defined on 
$[0,T]$ with values in $M_d\left(\R\right)$ such that for all $t \in [0,T]$, $\sigma\left(t\right)$ is invertible. We will set $\Sigma := \sigma \sigma^\top$.
\smallbreak
\noindent Let $\widetilde{b} : [0,T]\times\R^d \mapsto \R^d$ and $b_c : [0,T]\times\R^d \times \R^d \times S^{++}_d\left(\R\right) \mapsto \R^d$ defined by
\begin{equation} \label{E42}
b_c : \left(t,x,m,Q\right) \mapsto \Sigma(t) Q^{-1}\left(x-m\right), \ \widetilde{b} : \left(t,x\right) \mapsto a\left(t\right)x + c\left(t\right),
\end{equation}
where $a, c$ were defined at Section \ref{R41}.  Let  $H: [0,T]\times\R^d\times\R\times\R^d \rightarrow \R$
and $g: \R^d \rightarrow \R$. The goal of this subsection is to provide a probabilistic representation of
 viscosity solutions, being continuous in time and continuously differentiable in space, of the semilinear PDE

\begin{equation}
\label{SLPDE}
\left\{
\begin{array}{lll}
\partial_t v\left(t,x\right) + 
\frac{1}{2}Tr\left(\Sigma\left(t\right)\nabla^2_x v\left(t,x\right)\right)
+ H \left(t,x,v\left(t,x\right),\nabla_xv\left(t,x\right)\right) = 0,\ \left(t,x\right) \in [0,T[\times\R^d \\
v\left(T,\cdot\right) = g.
\end{array}
\right.
\end{equation}
\noindent 
\smallbreak
\noindent 
To formulate the result we consider the following assumption.
\begin{ass}\label{ass_g}
$g$ is continuous and has polynomial growth. 
\end{ass}

\noindent Let $\nu$ be a Gaussian Borel probability on $\R^d$ 
with mean ${\bar m}^\nu$  and covariance ${\bar Q}^\nu$. 
Let $t \mapsto m^\nu(t)$ defined in
\eqref{m_explicit},
$t \mapsto Q^\nu(t)$ be given by \eqref{Q_explicit}
and suppose that $\nu$ fulfills Assumption \ref{ass_nu}.
\smallbreak
\noindent We fix a filtered probability space $\left(\Omega, \mathcal{F},\left(\mathcal{F}_t\right)_{t\in[0,T]},\mathbb{P}\right)$ on which are defined a $d$-dimensional Brownian motion $\beta$ and a random vector $\xi_0$ distributed according to $\nu$ and independent of $\beta$. 
\smallbreak


\noindent Let $\xi$ be the unique strong solution of 
\begin{equation} \label{Rev-SDE}
 \displaystyle \xi_t = \xi_0 - \int^{t}_{0}\widetilde{b}\left(T-s,\xi_s\right) + b_c\left(T-s,\xi_s,m^\nu\left(T-s\right),Q^\nu\left(T-s\right)\right)ds + \int^{t}_{0}\sigma\left(T-s\right)d\beta_s, \ t \in [0,T[.
\end{equation}
\begin{rem}\label{StrongEx}
\noindent  \eqref{Rev-SDE} admits a unique strong solution on $[0,T[$ since 
its drift is affine with time-dependent continuous coefficients.
\end{rem}
\begin{lem} \label{HP_lemma}
\begin{enumerate}
\item The process $\widehat{\xi} := \xi_{T-\cdot}$ solves the SDE
\begin{equation}\label{direct_OU}
X_t = X_0 + \int^{t}_{0}\widetilde{b}\left(s,X_s\right)ds + \int^{t}_{0}\sigma\left(s\right)dW_s, \ t \in [0,T[, 
\end{equation}
where $W$ is an $\mathcal{F}^{\widehat \xi}$-Brownian motion independent of
$X_0 \sim \mathcal{N}\left(m^\nu\left(0\right),Q^\nu\left(0\right)\right)$.
\item $\widehat{\xi}$ extends continuously to $[0,T]$. 
\end{enumerate}
\end{lem}

\begin{proof}
  \begin{description}
  \item{i)}
    The SDE \eqref{direct_OU} admits in particular existence 
  in law. Let $X$ be a solution of  \eqref{direct_OU}.
To prove the first statement, we first show that the laws of
 $\widehat{\xi}$ and $X$ coincide.

 For this it is enough to prove that $\widehat{X} = X_{T-\cdot}$ and the solution $\xi$
 of \eqref{Rev-SDE} are identically distributed.
By Problem 6.1 in Chapter 5 in \cite{karatshreve}
and by uniqueness of the ODE \eqref{ODE_m}
(resp. \eqref{ODE_Q})
with initial condition 
$m^\nu\left(0\right)$
  (resp. $Q^\nu\left(0\right)$),
we get
$ \mathbb{E}\left(X_t\right) = m^\nu(t)$ and $\mathrm{Cov}\left(X_t\right) = Q^\nu(t)$ for all $t \in [0,T]$.
 By Problem 6.2, Chapter 5 in \cite{karatshreve})
 $X$  is a Gaussian process so
\begin{equation} \label{ENormXi}
  \widehat{\xi}_t \sim \mathcal{N}\left(m^\nu\left(t\right),Q^\nu\left(t\right)\right), \ t \in [0,T].
  \end{equation}
By  \eqref{ENormXi} and Theorem 2.1 
in \cite{haussmann_pardoux}, $\widehat{X}$ is a solution (in law) 
of \eqref{Rev-SDE}  on $ [0,T[$.
Pathwise uniqueness for \eqref{Rev-SDE} implies uniqueness in law on $[0,T[$ and the first statement of Lemma \ref{HP_lemma} is established.
  \item{ii)} We proceed now with the proof of the first statement.
    Let $X$ be a solution of \eqref{direct_OU}, so that
    we know that $W$ is a Brownian motion independent of $X_0$.
     On the other hand the process
    $$ M^X_t:= X_t - X_0 - \int_0^t \tilde b (u,X_u)du, \ t \in [0,T].$$
     is an $\shf^X$-martingale
    with quadratic variation
    $$ [M^X,(M^X)^\top] = \int_0^\cdot \Sigma(u) du.$$
  We have 
  \begin{equation} \label{EM11}
W \equiv \int_0^\cdot \sigma^{-1}(u) d M^{X}_u.
\end{equation}
Since $[W,W^\top]_t \equiv t I_d$, by L\'evy's characterization
theorem, $W$ is a standard $(\shf^X_t)$-Brownian motion.
%
     We set
      $$ M^{\hat \xi}_t := \hat \xi_t - \hat \xi_0 - \int_0^t \tilde b (u,\hat \xi_u)du, \ t \in [0,T],$$
      and we denote $W^{\hat \xi} := \int_0^\cdot \sigma^{-1}(u) d M^{\hat \xi}_u$.
Taking i) into account and the fact that
$\hat \xi$ and $X$ have the same  law,
then $W$ and $W^{\hat \xi}$ are identically distributed and so
$W^{\hat \xi}$ is an $\shf^{\hat \xi}$-Brownian motion.
Moreover the couple $(X_0, W)$ has the same distribution
as $(\hat \xi_0, W^{\hat \xi})$.
Consequently  $W^{\hat \xi}$ is an
$\shf^{\hat \xi}$-standard Brownian motion (independent of
$\hat \xi_0$) and
     the statement 1. follows.

    \item{iii)} It remains to prove the second statement. For this we show 
\begin{equation} \label{finiteMean}
\mathbb{E}\left(\int^{T}_{0}\left|b_c\left(s,\widehat{\xi}_s,m^\nu\left(s\right),Q^\nu\left(s\right)\right)\right|ds\right) < \infty.
\end{equation}
\noindent On the one hand, for all $t \in ]0,T]$,
\begin{align*}
\left|b_c\left(t,\widehat{\xi}_t,m^\nu\left(t\right),Q^\nu\left(t\right)\right)\right| &{}= \left|\Sigma\left(t\right)\sqrt{Q^\nu\left(t\right)^{-1}}\sqrt{Q^\nu\left(t\right)^{-1}}\left(\widehat{\xi}_t-m^\nu\left(t\right)\right)\right|\\
&{} \leq \left|\left|\Sigma\right|\right|_{\infty}\sqrt{\left|\left|Q^\nu\left(t\right)^{-1}\right|\right|}\left|\sqrt{Q^\nu\left(t\right)^{-1}}\left(\widehat{\xi}_t-m^\nu\left(t\right)\right)\right|\\
&{} = \frac{\left|\left|\Sigma\right|\right|_{\infty}}{\sqrt{\left|\left|Q^\nu\left(t\right)\right|\right|}}\left|\sqrt{Q^\nu\left(t\right)^{-1}}\left(\widehat{\xi}_t-m^\nu\left(t\right)\right)\right|,
\end{align*}
\noindent remembering that $Q^\nu\left(t\right)$ belongs to $S^{++}_d\left(\R\right)$.
\smallbreak
\noindent On the other hand, by \eqref{ENormXi}
$\left|\sqrt{Q^\nu\left(t\right)^{-1}}\left(\widehat{\xi}_t-m^\nu\left(t\right)\right)\right| \sim \left|Z\right|$ 
where $Z \sim \mathcal{N}\left(0,I_d\right)$. Then, \eqref{finiteMean} is verified if we show
\begin{equation} \label{finiteMeanBis}
\int^{T}_{0}\frac{1}{\sqrt{\left|\left|Q^\nu\left(t\right)\right|\right|}}dt < \infty.
\end{equation}
\noindent If $Q^\nu\left(0\right) = 0$, then for all $t \in ]0,T]$,
  for all $t \in ]0,T]$,  Remark \ref{rem_nondegen} implies 
\begin{equation*}
\frac{Q^\nu\left(t\right)}{t} = \sha\left(t\right)\left(\frac{1}{t}\int^{t}_{0}\sha\left(s\right)^{-1}\Sigma\left(s\right)
\left(\sha\left(s\right)^{-1}\right)^\top ds\right)\sha\left(t\right)^\top \underset{t\to 0}{\longrightarrow} \Sigma\left(0\right).
\end{equation*}
\noindent If $Q^\nu\left(0\right) \neq 0$,  then for all $]0,T]$,
again Remark \ref{rem_nondegen} yields
\begin{equation*}
\frac{\left|\left|Q^\nu\left(t\right)\right|\right|}{t} \geq \left| \ \left|\left|\sha\left(t\right)\frac{Q^\nu\left(0\right)}{t}\sha\left(t\right)^\top\right|\right|-\left|\left|\sha\left(t\right)\left(\frac{1}{t}\int^{t}_{0}\sha\left(s\right)^{-1}\Sigma\left(s\right)
\left(\sha\left(s\right)^{-1}\right)^\top ds\right)\sha\left(t\right)^\top\right|\right| \ \right| \underset{t\to 0}{\longrightarrow} + \infty,
\end{equation*}
\noindent where we have also used 
the fact $\sha\left(0\right) = I_d$ and the fact $\frac{1}{t}\int^{t}_{0}\sha\left(s\right)^{-1}\Sigma\left(s\right)\left(\sha\left(s\right)^{-1}\right)^Tds$ tends to $\Sigma\left(0\right)$ as $t$ tends to $0$ thanks to the continuity of $\Sigma,\sha^{-1}$ on $[0,T]$.
\smallbreak
\noindent Hence, for all $t \in ]0,T]$,
\begin{equation}
\lim\limits_{t \to 0} \frac{\sqrt{t}}{\sqrt{\left|\left|Q^\nu\left(t\right)\right|\right|}} = 
\begin{cases}
\frac{1}{\sqrt{\left|\left|\Sigma\left(0\right)\right|\right|}}, \ \rm{if}\  Q^\nu\left(0\right) = 0 \\
0, \ \rm{otherwise}.\\
\end{cases}
\end{equation}
\noindent This yields \eqref{finiteMeanBis} which implies  \eqref{finiteMean};
consequently the solution $X$ of \eqref{Rev-SDE} prolongates to $t = T$
and item 2. is proved.
\end{description}

\end{proof}
\noindent Though, this will not be exploited in the algorithm proposed at Section~\ref{S5}, 
it is interesting to note that 
the process $\xi$ introduced in \eqref{Rev-SDE} can  also be seen as the
solution of a McKean SDE.
Proposition \ref{McKean} below shows that \eqref{Rev-SDEIntro}
admits existence and uniqueness if and only if
Assumption \ref{ass_nu} is verified.
In particular we have the following.
\begin{prop} \label{McKean}
 \begin{enumerate}
 \item There is at most one  solution $(\xi, m, Q)$ of \eqref{Rev-SDEIntro}.
 \item Suppose the validity of Assumption \ref{ass_nu}. Let $\xi$ be the unique solution
   of \eqref{Rev-SDE}. Then
$(\xi, m^\nu, Q^\nu)$
   is a solution of \eqref{Rev-SDEIntro}.
\end{enumerate}
   \end{prop}
\begin{proof}
\begin{enumerate}
\item Let $\left(\xi,m,Q\right)$ be a solution of \eqref{Rev-SDEIntro}. By definition, $\xi$ solves an SDE of type \eqref{direct_OU} replacing $a$ by 
$a^\Sigma : s \mapsto - a\left(T-s\right) - \Sigma\left(T-s\right)Q\left(T-s\right)^{-1}$ and $c$ by $c^\Sigma : s \mapsto - c\left(T-s\right) + \Sigma\left(T-s\right)Q\left(T-s\right)^{-1}m\left(T-s\right)$.
\smallbreak
\noindent By Problem 6.1 Section 5 in \cite{karatshreve}, the function $t \mapsto \mathbb{E}\left(\xi_t\right) (= m\left(T-t\right))$ (resp. $t \mapsto {\rm Cov}\left(\xi_t\right) (= Q\left(T-t\right))$) solves the first line of \eqref{ODE_m} (resp. \eqref{ODE_Q}) replacing $a$ by $a^\Sigma$ and $c$ by $c^\Sigma$. Then, the following identities hold for all $t \in ]0,T]$:
\begin{equation} \label{Id-rev-m}
m\left(T-t\right) = \mathbb{E}\left(\xi_0\right) - \int^{t}_{0}a\left(T-s\right)m\left(T-s\right) + c\left(T-s\right)ds,
\end{equation}
\begin{equation} \label{Id-rev-Q}
Q\left(T-t\right) = {\rm Cov}\left(\xi_0\right) - \int^{t}_{0}Q\left(T-s\right)a\left(T-s\right)^\top + a\left(T-s\right)Q\left(T-s\right)ds - \int^{t}_{0}\Sigma\left(T-s\right)ds, 
\end{equation}
\noindent remarking that
\begin{equation*}
a^\Sigma\left(t\right)m\left(T-t\right) + c^\Sigma\left(t\right) = -a\left(T-t\right)m\left(T-t\right) - c\left(T-t\right),
\end{equation*}
\begin{equation*}
Q\left(T-t\right)a^{\Sigma}\left(t\right)^\top +  a^{\Sigma}\left(T-t\right) Q\left(t\right) = - Q\left(T-t\right)a\left(T-t\right)^\top -  a\left(T-t\right) Q\left(T-t\right) - 2\Sigma\left(T-t\right).
\end{equation*}

\smallbreak
\noindent Applying the change of variable $t \mapsto T-t$ in identities \eqref{Id-rev-m} and \eqref{Id-rev-Q}, we show that $m$ (resp. $Q$) solves
the  backward ODE \eqref{ODE_m} (resp. \eqref{ODE_Q}),
which is well-posed. We recall that $\xi_0$ is distributed
according to $\nu$.
Then, $m=m^\nu$ and $Q = Q^\nu$, see the beginning of Section \ref{R41}.
 As a consequence, $\xi$ solves \eqref{Rev-SDE} and is uniquely determined thanks to Remark \ref{StrongEx}. This shows the validity of item 1.

\item Let $\xi$ be the unique solution of \eqref{Rev-SDE}. Then,
  the time-reversed process $\widehat{\xi}$ solves \eqref{direct_OU} and $\xi_T \sim \mathcal{N}\left(m^\nu\left(0\right),Q^\nu\left(0\right)\right)$, thanks to item 1. of Lemma \ref{HP_lemma}. Now, by  Remark \ref{RmCov}, we have  $\mathbb{E}\left(\widehat{\xi}_t\right) = m^\nu\left(t\right)$, ${\rm Cov}\left(\widehat{\xi}_t\right) = Q^\nu\left(t\right)$ for all $t \in [0,T[$. This concludes
  the proof of item 2.
\end{enumerate}
\end{proof}

\begin{rem} \label{RIOR}
  \begin{enumerate}
  \item In \cite{LucasOR} we have discussed existence
    and uniqueness of more  general McKean problems
    involving the densities of the marginal laws instead of
    expectation and covariance matrix, where the solution
    is the time-reversal of some (not necessarily Gaussian)
    diffusion.
  \item In particular, in Section 4.5 of \cite{LucasOR}
    we have investigated existence and uniqueness
    of
    \begin{equation}\label{MKIntro}
\begin{cases}
\displaystyle Y_t = Y_0 - \int^{t}_{0}\tilde b\left(T-r,Y_r\right)dr +
 \int^{t}_{0}\left\{\frac{\mathop{div_y}\left(\Sigma_{i.}\left(T-r\right)p_{r} \left(Y_r\right)\right)}{p_{r}\left(Y_r\right)}\right\}_{i\in[\![1,d]\!]}dr + \int^{t}_{0} \sigma\left(T-r\right)d\beta_r, \\
p_t \ \rm{density\ law\ of} \ {\bf p}_t =  \rm{law\ of} \ Y_t, t \in ]0,T[,
\\
Y_0 \sim  {\bf p_T} =  \nu,
\end{cases}
\end{equation}
where $\beta$ is a $m$-dimensional Brownian motion and
$\Sigma = \sigma \sigma^\top$, whose solution is the couple $(Y,{\bf p}).$
Moreover, when the solution exists, there
is a probability-valued function ${\bf u}$
    defined on $[0,T]$ solution
    of the Fokker-Planck equation
   \begin{equation} \label{EDPTerm0Bis}
\left \{
\begin{array}{lll}
\partial_t {\bf u} &=& \frac{1}{2} 
\displaystyle{\sum_{i,j=1}^d} \partial_{ij}^2 \left( (\sigma \sigma^{\top})_{i,j}(t)
   {\bf u} \right) - div \left( \tilde b(t,x) {\bf u} \right)\\
{\bf u}(T) &=& \nu.
\end{array}
\right .
\end{equation}
 \item Suppose that
   $\nu$ is a Gaussian law on $\R^d$.
   It is possible  to show that
   Assumption \ref{ass_nu} is equivalent
   to the existence of a probability-valued
   solution  ${\bf u}$ of \eqref{EDPTerm0Bis}.
   In this case the McKean problems
   \eqref{Rev-SDEIntro} and \eqref{MKIntro}
   are equivalent. In particular the component $Y$
   of the solution of \eqref{MKIntro} is Gaussian.
    \end{enumerate}
  \end{rem}

\noindent We continue with a preliminary lemma.
Let $W$ be a Brownian motion.
For each $\left(s,x\right) \in [0,T]\times\R^d$, $X^{s,x}$
will denote below the process
$$ X^{s,x}_t :=  x+ \int_s^t \sigma(r)dW_r, \ t \in [s,T].$$
\begin{lem} \label{Abstract-Lemma}
\noindent Suppose the validity of Assumption 
\ref{ass_g}.
Let $v : [0,T]\times\R^d \rightarrow \R$ 
of class $ \shc^{0,1}\left([0,T],\R^d\right)$, with polynomial growth and such that the function
 $H^v : \left(t,x\right) \mapsto H\left(t,x,v\left(t,x\right),\nabla_xv\left(t,x\right)\right)$ is continuous with polynomial growth. Then, the following assertions are equivalent.
\begin{enumerate}
\item $v$ is a viscosity solution of \eqref{SLPDE}.
\item For each $\left(s,x\right)\in[0,T]\times\R^d$,
\begin{equation} \label{Id-vflow}
v\left(s,x\right) = \mathbb{E}\left(\int^{T}_{s}H\left(r, X^{s,x}_r,v\left(r,X^{s,x}_r\right),\nabla_xv\left(r,X^{s,x}_r\right)\right)dr + g\left(X^{s,x}_T\right)\right).
\end{equation}
\item $v$ is of class $\shc^{1,2}\left([0,T[,\R^d\right)$ and is a
 (classical) solution of \eqref{SLPDE}.
\end{enumerate}
\end{lem}
\begin{proof}
  \noindent Let $v$ as in the lemma statement.
\begin{description}
\item{a)}  
\noindent 
We set
\begin{equation} \label{E310bis}
 w^v(s,x) := \mathbb{E}\left(g(X^{s,x}_T) + \int_s^T H^v(r, X^{s,x}_r) dr\right),  \ \left(s,x\right) \in [0,T[\times\R^d.
\end{equation}
We show first that $w^v$ is a (classical) solution in
$\shc^{1,2}\left([0,T[,\R^d\right) \cap \shc^0\left([0,T]\times\R^d\right)$
    with polynomial growth of the linear PDE
\begin{equation} \label{lin_heat_PDE}
\left\{
\begin{array}{lll}
\partial_t w\left(t,x\right) + \frac{1}{2} Tr[\sigma \sigma^\top(t) \nabla_x^2 w\left(t,x\right)]  + H^v\left(t,x\right) = 0, \ \left(t,x\right) \in [0,T[\times\R^d
\\
w(T,\cdot)= g.
\end{array}
\right.
\end{equation}
  Indeed $w^v$ can be rewritten as  
\begin{equation} \label{E310}
w\left(s,x\right) = \int_{\R^d}g\left(z\right)p_T\left(s, z - x\right)dz
 + \int^{T}_{s}\int_{\R^d}H^v\left(r,z\right)p_r\left(s,z - x\right)dzdr, \ \left(s,x\right) \in [0,T[\times\R^d,
\end{equation}
\noindent where for each $r \in [0,T]$, $s \in [0,r[$,
$p_r\left(s,\cdot\right)$ is the density of the r.v.
 $\int_s^r \sigma(u) dW_u$, i.e. a Gaussian r.v. with mean zero
 and covariance $\int^{r}_{s}\Sigma\left(u\right)du$.
Moreover, it is well-known, see e.g. Remark 3.2 in  \cite{DGRClassical}, 
that for each $r \in [0,T]$, $p_r: [0,r[ \times \R^d \rightarrow \R$ is a smooth solution of
\begin{equation} \label{eq EDDP}
\partial_{t} p_r(t,z) + \frac{1}{2} Tr\left(\Sigma\left(t\right)\nabla^2_x p_r\left(t,z\right)\right) = 0, \left(t,z\right) \in [0,r[\times\R^d.
\end{equation}
 Consequently, by usual integration theorems allowing to commute
 derivation and integrals, one shows \eqref{lin_heat_PDE}.
\item{b)} Consequently $w^v$ is a viscosity solution
  \eqref{lin_heat_PDE}.
\item{c)} If 1. holds then $v$ is also a viscosity solution of
  \eqref{lin_heat_PDE}.
By point 1. of Remark \ref{uniq_visco_rem},
equation \eqref{lin_heat_PDE} admits at most one continuous viscosity solution with polynomial growth. So $v = w^v$ which means 2.
\item{d)} If 2. holds then $v = w^v$ and by b) $v$ is
  a viscosity solution of \eqref{lin_heat_PDE} and
  therefore of \eqref{SLPDE}.
\item{e)} 3. implies obviously 1. Viceversa, if item 1.
holds, a) implies that $w^v$ is a classical solution 
of \eqref{lin_heat_PDE}; b) and the uniqueness
of viscosity solutions for previous linear equation
implies $v= w^v$ and finally item 3.
  \end{description}

\end{proof}

\noindent We state now the announced representation result.
\begin{thm}\label{SLPDE-Rep}
Suppose the validity of Assumption \ref{ass_g}.
Let $\nu$ be a Gaussian probability fulfilling
Assumption \ref{ass_nu} with associated functions
$m^\nu$ and $Q^\nu$.

\noindent Let $v \in \shc^{0,1}\left([0,T],\R^d;\R\right)$ with polynomial growth and such that $H^v : \left(t,x\right) \mapsto H\left(t,x,v\left(t,x\right),\nabla_x v\left(t,x\right)\right)$ is continuous with polynomial growth.  Then, $v$ 
 is a viscosity solution of \eqref{SLPDE}  if and only if for all $t \in [0,T]$
\begin{equation} \label{RepFormula}
\left\{
\begin{array}{lll}
 \displaystyle \xi_t = \xi_0 - \int^{t}_{0}\widetilde{b}\left(T-s,\xi_s\right) + b_c\left(T-s,\xi_s,m^\nu\left(T-s\right),Q^\nu\left(T-s\right)\right)ds + \int^{t}_{0}\sigma\left(T-s\right)d\beta_s, \\
\xi_0 \sim \nu, \\
\displaystyle v\left(t,\widehat{\xi}_t\right) = \mathbb{E}\left(\int^{T}_{t}
H\left(s,\widehat{\xi}_s,v\left(s,\widehat{\xi}_s\right),
\nabla_xv\left(s,\widehat{\xi}_s\right)\right) - \left<\widetilde{b}\left(s,\widehat{\xi}_s\right),\nabla_x v\left(s,\widehat{\xi}_s\right)\right>ds + g\left(\widehat \xi_T \right)\middle|\widehat{\xi}_t\right).
\end{array}
\right.
\end{equation}
\end{thm}
\begin{rem} \label{RSLPDE-Rep}
  The affine drift $\widetilde{b}$ remains a degree of freedom of the
representation. In Section \ref{S5}, in the framework of 
the Hamilton-Jacobi-Bellman PDEs are given elements 
to choose rationally $\widetilde{b}$.
\end{rem}
\begin{rem} \label{uniq_visco_rem}
 We remark that previous representation \eqref{RepFormula}
is valid even if uniqueness does not hold for the
semilinear  PDE \eqref{SLPDE}. In that case even the equation
\eqref{RepFormula} does not admit uniqueness.
However, we provide below 
  some typical situations
  for which \eqref{SLPDE} admits at most one viscosity solution, within
 different classes of solutions.
\begin{enumerate}
\item Suppose the validity of Assumption \ref{ass_g}. Suppose also that $H$ is continuous with polynomial growth in $x$ and linear growth in $\left(y,z\right)$. In addition, we suppose that
  $H$ is Lipschitz in $\left(y,z\right)$ uniformly in $\left(t,x\right)$ and suppose that for all $R > 0$, there exists $m_R : \R \rightarrow \R^+$, tending to $0$ at $0^+$ such that
\begin{equation*}
\left|H\left(t,x',y,z\right) - H\left(t,x,y,z\right)\right| \leq m_R\left(\left|x'-x\right|\left(1 + \left|z\right|\right)\right), 
\end{equation*}
\noindent for all $t \in [0,T]$, $z \in \R^d$ and $\left|x\right|,\left|x'\right|,\left|y\right| \leq R$. Then, by Theorem 5.1 in \cite{PardPradRao},
implies that \eqref{SLPDE} admits at most one continuous viscosity solution with polynomial growth. 
In fact that theorem
states uniqueness even in a wider class of solutions.
\item The
 first theorem in \cite{LionsSouganidisJensen} formulates a uniqueness result
 in a suitable class of bounded uniformly continuous solutions.
 Alternative assumptions are available to ensure uniqueness
in different classes of  unbounded functions, 
for fully non-linear parabolic Cauchy problems.
See for instance Corollary 2 in \cite{ishiiKobayasi}, 
Theorem 3.1 in \cite{NUNZIANTE}, 
 \cite{userguide_visco}, \cite{Ishii_uniq}.
\end{enumerate}
\end{rem}

\begin{prooff} \ (of Theorem \ref{SLPDE-Rep}).
\smallbreak
\noindent Let $v$ as in the statement.
\begin{enumerate}

\item Lemma \ref{HP_lemma} implies that there exists an $\mathcal{F}^{\widehat{\xi}}$-Brownian motion $W$ such that, under $\mathbb{P}$,
\begin{equation} \label{OU-xi}
\widehat{\xi}_t = \widehat{\xi}_0 + \int^{t}_{0}\widetilde{b}\left(s,\widehat{\xi}_s\right)ds + \int^{t}_{0}\sigma\left(s\right)dW_s, \ t \in [0,T],
\end{equation}
\noindent where $\widehat{\xi}_0 \sim \mathcal{N}\left(m^\nu\left(0\right),Q^\nu\left(0\right)\right)$.
In particular 
\begin{equation} \label{EexpXi}
\mathbb{E}\left(\sup_{t\in[0,T]} \left|\widehat{\xi}_s\right|^p\right) < \infty, \ \forall p \ge 1.
\end{equation}
This, together with Assumption \ref{ass_g} and the polynomial growth of $H^v$
also imply that the r.v.
\begin{equation*}\label{ESQ}
\int^{T}_{0}
H\left(s,\widehat{\xi}_s,v\left(s,\widehat{\xi}_s\right),
\nabla_xv\left(s,\widehat{\xi}_s\right)\right) - \left<\widetilde{b}\left(s,\widehat{\xi}_s\right),\nabla_x v\left(s,\widehat{\xi}_s\right)\right>ds + g(\widehat \xi_T) 
\end{equation*}
\noindent is square integrable. 
\item We give now an equivalent formulation of \eqref{RepFormula} using a change of probability measure. 
\smallbreak
\noindent  
We set
 $ L_s : = \sigma\left(s\right)^{-1}\widetilde{b}\left(s,\widehat{\xi}_s\right),
 \ s \in [0,T]$.
 We denote by $\mathbb{Q}$, the probability equivalent to $\mathbb{P}$ on $\shf^{\widehat{\xi}}_T$
 defined by $\frac{d\mathbb{Q}}{d\mathbb{P}} = \mathcal{E}\left(-\sum^{d}_{i=1}\int^{\cdot}_{0}L^i_sdW^i_s\right)_T$, being well-defined thanks to Lemma \ref{Girsanov_OU}.

\smallbreak 
\noindent The goal is to show that $v$ fulfills \eqref{RepFormula} if and only if it fulfills for all $t \in [0,T]$
\begin{equation}\label{RepFormulaBis}
v\left(t,\widehat{\xi}_t\right) = \mathbb{E}_{\mathbb{Q}}\left(\int^{T}_{t}
H\left(s,\widehat{\xi}_s,v\left(s,\widehat{\xi}_s\right),\nabla_xv\left(s,\widehat{\xi}_s\right)\right)ds + g\left(\widehat{\xi}_T\right)\middle|\widehat{\xi}_t\right).
\end{equation}
We remark that, 
\begin{equation} \label{OU-xiBis}
\widehat{\xi}_t = \widehat{\xi}_0  + \int^{t}_{0}\sigma\left(s\right)d\widetilde{W}_s,
 \ t \in [0,T],
\end{equation}
where
\begin{equation}\label{ETildeW} 
\widetilde{W} := W + \int^{\cdot}_{0}L_s ds,
\end{equation}
which is a Brownian motion under $\mathbb{Q}$ thanks to Girsanov's Theorem 5.1 in \cite{karatshreve}.
By item 1. 
$$ \int^{T}_{0}
H\left(s,\widehat{\xi}_s,v(s,\widehat{\xi}_s),\nabla_xv\left(s,\widehat{\xi}_s\right)\right)ds + 
g(\widehat{\xi}_T),$$
is obviously also square integrable under $\Q$.

\noindent We set
 $H_s  := H\left(s,\widehat{\xi}_s,v\left(s,\widehat{\xi}_s\right),\nabla_xv\left(s,\widehat{\xi}_s\right)\right), \ s \in [0,T],$ for the sake of brevity. 

\noindent We remark first that for each given $s \in [0,T]$,
\begin{equation} \label{EToBegin}
\left<\widetilde{b}\left(s,\widehat{\xi}_s\right),\nabla_x v\left(s,\widehat{\xi}_s\right)\right>  = \left<\sigma\left(s\right)\sigma\left(s\right)^{-1}\widetilde{b}\left(s,\widehat{\xi}_s\right),\nabla_x v\left(s,\widehat{\xi}_s\right)\right>= \left<L_s,\sigma\left(s\right)^\top\nabla_xv\left(s,\widehat{\xi}_s \right)\right>. 
\end{equation} 
\smallbreak
\noindent Then, \eqref{EToBegin} combined with the Markov property of $\widehat{\xi}$ implies that \eqref{RepFormula} is equivalent to 
\begin{equation*}\label{MarkovRep}
\displaystyle v\left(t,\widehat{\xi}_t\right) = \mathbb{E}\left(\int^{T}_{t}
\left(H_s - \left<L_s,\sigma\left(s\right)^\top\nabla_xv\left(s,\widehat{\xi}_s\right)\right> \right) ds + g\left(\widehat{\xi}_T\right)\middle|\mathcal{F}^{\widehat{\xi}}_t\right),
\end{equation*}
\noindent which can be rewritten
\begin{equation*}\label{MakovRep_bis}
\displaystyle v\left(t,\widehat{\xi}_t\right) = M_t - 
\int^{t}_{0} \left(H_s - \left<L_s,\sigma\left(s\right)^\top\nabla_xv\left(s,\widehat{\xi}_s\right)\right>\right)ds,
\end{equation*}
\noindent where $M$ is the $\mathbb{P}$-martingale
\begin{equation} \label{P-mart}
M_t = \mathbb{E}\left(\int^{T}_{0}H_s - \left<L_s,\sigma\left(s\right)^\top\nabla_xv\left(s,\widehat{\xi}_s\right)\right>ds + g\left(\widehat{\xi}_T\right)\middle|\mathcal{F}^{\widehat{\xi}}_t\right), \ t \in [0,T].
\end{equation}
\noindent Similarly, \eqref{RepFormulaBis} is equivalent to 
\begin{equation*}\label{MakovRep_ter}
\displaystyle v\left(t,\widehat{\xi}_t\right) = \bar{M}_t - \int^{t}_{0}H_sds,
\end{equation*}
\noindent where $\bar{M}$ is the $\mathbb{Q}$-martingale 
\begin{equation} \label{Q-mart}
\bar{M}_t = \mathbb{E}_{\mathbb{Q}}\left(\int^{T}_{0}H_sds + g\left(\widehat{\xi}_T\right)\middle|\mathcal{F}^{\widehat{\xi}}_t\right), \ t \in [0,T].
\end{equation}
\noindent To show the aforementioned equivalence, it suffices now to show 
\begin{equation*} \label{Id-mart}
   \bar{M}_t - M_t = \int^{t}_{0}\left<L_s,\sigma\left(s\right)^\top\nabla_xv\left(s,\widehat{\xi}_s\right)\right>ds, \ t \in [0,T].
\end{equation*}
\noindent On the one hand, Theorem 1.7 Chapter 8 in \cite{RevuzYorBook} implies that the process $\widetilde{M} := M + \sum^{d}_{i=1}[M,\int^{\cdot}_{0}L^i_sdW^i_s]$ is a $\mathbb{Q}$-local martingale. On the other hand, for each $i \in [\![1,d]\!]$ by Proposition 3.10 in \cite{gr} we have
\begin{align*}
[M,\int^{\cdot}_{0}L^i_sdW^i_s] &= [v\left(\cdot,\widehat{\xi}\right),\int^{\cdot}_{0}L^i_s dW^i_s] \\
& = \int^{\cdot}_{0}L^i_s\left(\sigma\left(s\right)^\top\nabla_xv\left(s,\widehat{\xi}_s\right)\right)_ids,
\end{align*}
\noindent combining \eqref{Q-mart} with the usual properties of covariation
for semimartingales.
This means that
\begin{equation*} \label{Id-mart-interm}
\widetilde{M} = M + \int^{\cdot}_{0}\left<L_s,\sigma\left(s\right)^\top\nabla_xv\left(s,\widehat{\xi}_s\right)\right>ds
\end{equation*}
\noindent is a $\mathbb{Q}$-local martingale. Now, 
\begin{align*}
\widetilde{M}_T &= M_T + \int^{T}_{0}\left<L_s,\sigma\left(s\right)^\top\nabla_xv\left(s,\widehat{\xi}_s\right)\right>ds \\
&{} = \int^{T}_{0}H_sds + g\left(\widehat{\xi}_T\right),
\end{align*}
\noindent thanks to \eqref{P-mart}. Since $\bar{M}$ and $\widetilde{M}$ are $\mathbb{Q}$-local martingales being equal at $t=T$, we have $\bar{M} = \widetilde{M}$. This shows the validity of point 2.
\smallbreak

\item 
For each $\left(s,x\right)\in[0,T]\times\R^d$, we set $X^{s,x} := x + \int^{\cdot}_{s}\sigma\left(r\right)d\widetilde{W}_r$ where $\widetilde{W}$ is the $\mathbb{Q}$-Brownian motion defined in \eqref{ETildeW}. Associated with $v$, we consider the continuous function
\begin{equation*}\label{vPsi}
 w^v\left(t,x\right) :=  \mathbb{E}_{\mathbb{Q}}\left(\int^{T}_{t}H\left(r, X^{t,x}_r,v\left(r,X^{t,x}_r\right),\nabla_x v\left(r,X^{t,x}_r\right)\right)dr +
 g\left(X^{t,x}_T\right)\right), \ (t,x) \in [0,T] \times \R^d.
 \end{equation*}
 We observe that $v$ fulfills \eqref{RepFormulaBis} if and only if for
 all $\left(t,x\right) \in [0,T]\times\R^d$ 
\begin{equation} \label{E436} 
v(t,x) = w^v\left(t,x\right).
\end{equation}

Indeed this follows by the freezing lemma of the conditional
expectation, the fact that $\widehat \xi_t$ is independent
of the random field $\left(X^{t,x}_s\right)_{ t\leq s\leq T,x\in\R^d}$
and the flow property
\begin{equation*}\label{EFlow}
 X^{t,\widehat{\xi}_t}_s = \widehat{\xi}_s, s \in [t,T].
\end{equation*}
\item It remains to show that \eqref{E436} is satisfied if and only if $v$ is a viscosity solution of \eqref{SLPDE}. This is the object of Lemma \ref{Abstract-Lemma} applied under the probability $\mathbb{Q}$, in particular to the equivalence between item 1. and item 3.

\end{enumerate}
\end{prooff}

\section{Representation of stochastic control problems} \label{control_problem_section}

\setcounter{equation}{0}
\noindent Let us briefly recall the link between stochastic control and non-linear
 PDEs given by the Hamilton-Jacobi-Bellman (HJB) equation. 
We refer for instance to~\cite{gozzibook, Pham09, touzibook} for more details. 
\smallbreak
\noindent Let $A \subset \R^k$ compact and denote by $\sha_0$ the set of all $A$-valued progressively measurable processes $\left(\alpha_t\right)_{t \in [0,T]}$, namely the set of \textit{admissible controls}.
\smallbreak
\noindent We consider now \textit{state processes} $(X_{t}^{s,x,\alpha})_{s\leq t \leq T, \alpha \in \sha_{0}}$ starting at time $s \in [0,T]$ with value $x \in \R^d$, solutions of the controlled SDE
\begin{equation} \label{controlled_SDE}
dX_t = b\left(t,X_t,\alpha_t\right)dt + \sigma\left(t\right)dW_t, 
\end{equation}
\noindent where $W$ is a $d$-dimensional Brownian motion and $b : [0,T]\times\R^d\times A \mapsto \R^d$ is supposed to fulfill the following.
\begin{ass} \label{control_drift_ass}
\noindent The function $b$ is continuous and there exists $K \geq 0$ such that
\begin{equation*}
\left|b\left(t,x_2,a\right) - b\left(t,x_1,a\right)\right| \leq K\left|x_2-x_1\right|, \ \left(t,x_1,x_2,a\right) \in [0,T]\times \R^d \times \R^d \times A.
\end{equation*}
\end{ass}

\noindent Note that Assumption \ref{control_drift_ass} implies $b$ to have
 linear growth in space uniformly in time and in the control.
 Consequently, \eqref{controlled_SDE} starting at time $s$ with value $x$ admits a unique solution for each $\alpha \in \sha_{0}$, for each $\left(s,x\right) \in [0,T]\times\R^d$, 
 by the same arguments as in Theorem 3.1 in \cite{touzibook}.
\smallbreak
\noindent We also introduce the \textit{cost function} $J : [0,T]\times\R^d\times\sha_{0} \rightarrow \R$ defined by 
\begin{equation} \label{cost_function}
J(s,x,\alpha) :=
\mathbb{E} \left(g(X_{T}^{s,x,\alpha})+ \int _{s}^{T} f\big (r,X_{r}^{s,x,\alpha},\alpha_r\big )dr\right), \ \left(s,x,\alpha\right) \in [0,T]\times\R^d\times\sha_{0},
\end{equation}
\noindent where the function $f : [0,T]\times\R^d\times A \mapsto \R$ (\textit{running cost}) is supposed to fulfill what follows. 
\begin{ass} \label{control_costs_ass}
\noindent The function $f$ is continuous and there exists $m,M \geq 0$ such that
\begin{equation*}
\left|f\left(t,x,a\right)\right| \leq M\left(1 + \left|x\right|^m\right), \ \left(t,x,a\right) \in [0,T]\times\R^d \times A.
\end{equation*}
\end{ass}
\noindent Supposing the validity of Assumptions \ref{control_drift_ass}
and \ref{control_costs_ass} together with  Assumption \ref{ass_g}
on the function $g : \R^d \mapsto \R$ (\textit{terminal cost}),
we are interested in minimizing, over control processes $\alpha \in \sha_{0}$ the functions $\alpha \mapsto J\left(0,x,\alpha\right)$ for every $x \in \R^d$.
\smallbreak
\noindent To tackle this finite horizon stochastic control problem, the usual approach consists in introducing the associated \textit{value} (or \textit{Bellman}) function 
 $v : [0,T] \times \mathbb{R}^{d} \rightarrow \R$ representing the minimum expected costs, starting from any time $t \in [0,T]$ at any state $x \in \R^d$, i.e.
\begin{equation}
\label{eq:prob}
v(t,x):= \displaystyle {\inf_{\alpha \in \mathcal{A}_0} }
\, 
J\left(t,x,\alpha\right), \ \left(t,x\right) \in [0,T]\times\R^d \ .
\end{equation}
\noindent Note that the terminal condition is known, which fixes $v\left(T,\cdot\right)=g$, whereas $v\left(0,\cdot\right)$ corresponds to the solution of the original minimization problem. 
\smallbreak
\begin{rem} \label{RHamilt}
Suppose the validity of Assumptions \ref{ass_g}, \ref{control_drift_ass} and \ref{control_costs_ass}.
 \begin{enumerate} 
\item The function $v$ is continuous on $[0,T] \times \mathbb{R}^{d}$  and has polynomial growth, see Theorem 5. Chapter 3. in \cite{krylov}. 
\item
 The value function $v$ is a viscosity solution of the Hamilton-Jacobi-Bellman equation
\begin{equation}
\label{eq:HJB}
\left\{
\begin{array}{lll}
\partial_t v (t,x)+H(t,x,\nabla_x v(t,x)) +\frac{1}{2} Tr[\sigma \sigma^\top(t) \nabla_x^2 v(t,x)]=0, \ \left(t,x\right) \in [0,T[\times\R^d
\\
v(T,\cdot)=g,\\ 
\end{array}
\right .
\end{equation}
\noindent where $H$ denotes the real-valued function defined on $[0,T]\times \R^d\times\R^d$  by 
\begin{equation} 
\label{eq:H1}
H(t,x,\delta):=\inf_{ a\in A} \left \{f(t,x, a)+ \left<b(t,x, a),\delta\right> \right \}, \ \left(t,x,\delta\right) \in [0,T]\times \R^d \times \R^d,
\end{equation}
\noindent see for example Theorem 7.4 in \cite{touzibook}. 
\item  By  definition, it is obvious that $(x,z) \mapsto H(t,x,z)$ has polynomial growth  uniformly
with respect to $t$. It is also clear that $H$ is continuous.
\item Under Assumptions \ref{ass_g}, \ref{control_drift_ass} and
\ref{control_costs_ass}, the PDE \eqref{eq:HJB} admits at most one viscosity solution in the class of continuous solutions with polynomial growth, see Theorem {\rm II}.3 in \cite{Uniq_HJB_Lions}.
Since $v$ has polynomial growth, the value function
 $v$ is the unique viscosity solution of \eqref{eq:HJB} in the considered class.
\end{enumerate}
\end{rem}
\noindent We formulate below another assumption for the value function $v$. 
\begin{ass} \label{optimal_control}
$v$ is of class $\shc^{0,1}\left([0,T],\R^d\right)$
such that $\nabla_x v$ has polynomial growth.
\end{ass}
\begin{rem}\label{RControl}
  Under Assumption \ref{optimal_control}, using
Remark \ref{RHamilt} 3.
that the function
$(t,x) \mapsto H^v\left(t,x\right):= H\left(t,x,\nabla_x v \left(t,x\right)\right)$ is continuous 
with polynomial growth.
\end{rem}
\begin{rem} \label{RControlBis}
\begin{enumerate}
\item Assumption \ref{optimal_control}
  is not so restrictive, since whenever $g$ and $f$ are locally
  Lipschitz with polynomial growth gradient (in space), then $v$ is locally Lipschitz in the space variable. To prove this, it suffices to show that $J$ is locally Lipschitz in $x$ uniformly in $t$ and $\alpha$. A proof of this fact is given in Lemma \ref{loc_Lip_J} stated in the Appendix. 
\smallbreak
In that context, the value function $v$ is in particular absolutely continuous and for every $t \in [0,T]$, for almost every $x \in \R^d$, $v\left(t,\cdot\right)$ is differentiable and $\nabla_x v(t, \cdot)$ exists.
\item Suppose in addition that the functions $f$, $g$ and $b$ are of class $\shc^1$ (in the space variable) and  the validity of Assumption \ref{markov_optimal_control}.
 Then  $\nabla_x v(t, \cdot)$ has polynomial
  growth as we show below.
  Indeed, by usual dominated convergence arguments, we can show that for each $\left(t,\alpha\right)\in[0,T]\times\sha_0$, $x \mapsto J\left(t,x,\alpha\right)$ is differentiable with gradient
\begin{equation}\label{obj_func_grad}
\nabla_x J\left(t,x,\alpha\right) = \mathbb{E}\left(Y^{t,x,\alpha}_T\nabla_x g\left(X^{t,x,\alpha}_T\right) + \int^{T}_{t}Y^{t,x,\alpha}_r \nabla_xf\left(r,X^{t,x,\alpha}_r,\alpha_r\right)dr\right),
\end{equation}
\noindent where $Y^{t,x,\alpha}$ is the unique matrix-valued process fulfilling
\begin{equation*}
Y^{t,x,\alpha}_r = I_d + \int^{r}_{t}\nabla_x b\left(s,X^{t,x,\alpha}_s,\alpha_s\right)Y^{t,x,\alpha}_s ds, \ r \in[t,T],
\end{equation*}
where $\nabla_x b := \left(\partial_{x_j} b^i\right)_{\left(i,j\right)\in[\![1,d]\!]^2}$.
\smallbreak
\noindent  Combining what precedes with Lemma \ref{env_thm} stated in the Appendix, we deduce that for all $t \in [0,T]$, for almost every $x \in \R^d$
\begin{equation} \label{val_func_grad}
\nabla_x v \left(t,x\right) = \nabla_x J\left(t,x,\alpha^*\left(t,x\right)\right),
\end{equation}
\noindent where $\alpha^*$ is the Borel function introduced in Assumption \ref{markov_optimal_control}. In view of \eqref{obj_func_grad} and \eqref{val_func_grad}, $\nabla_xv$ has polynomial growth.
\end{enumerate}
\end{rem}

\begin{corro} \label{RepFormulaControl}
Let $\nu$ be a Gaussian probability measure
fulfilling Assumption \ref{ass_nu} with associated functions
$m^\nu$ and $Q^\nu$.
We suppose the validity of Assumptions \ref{ass_g}, \ref{control_drift_ass}, 
\ref{control_costs_ass}.
Among the functions $v:[0,T] \times \R^d \rightarrow \R$
fulfilling Assumption \ref{optimal_control},
the value function is the unique one which is solution of
\eqref{RepFormula}.
(In this framework $H$ only depends on $\nabla_x v$ and
not on $v$).
\end{corro}
\begin{proof} 
\
\noindent We recall that $H^v$ has polynomial growth by Remark 
\ref{RControl} 1.
Otherwise, on the one hand, by Remark \ref{RHamilt} and 
the direct implication in Theorem \ref{SLPDE-Rep}, 
  $v$ fulfills \eqref{RepFormula}.
 \noindent On the other hand, if a function  $v$ fulfills \eqref{RepFormula}
then, by the converse implication of Theorem \ref{SLPDE-Rep}
$v$ is a viscosity solution of \eqref{eq:HJB}.
By Remark \ref{RHamilt} 3., $v$ can only be the value function.

\end{proof}
\noindent We introduce a supplementary hypothesis on the value function $v$.
\begin{ass} \label{markov_optimal_control}
There exists a Borel function $\alpha^* : [0,T]\times\R^d \rightarrow A$ such that
\begin{equation*}
H\left(t,x,\nabla_x v\left(t,x\right)\right) = \left<b\left(t,x,\alpha^*\left(t,x\right)\right),\nabla_x v\left(t,x\right)\right> + f\left(t,x,\alpha^*\left(t,x\right)\right), \ \left(t,x\right) \in [0,T]\times\R^d.
\end{equation*}

\end{ass}
\noindent We state (and show below)  a \textit{verification} type result involving $\alpha^*$ without any further regularity assumptions on the value function.
That result is somehow classical, but it is not obvious 
to find it in the literature (see e.g. Chapter 5 of \cite{touzibook} or
\cite{gr1}),  with our assumptions.
So, for the consistency
of the paper we provide a proof. 
Note to begin that the Borel function $b^* : \left(t,x\right) \rightarrow b\left(t,x,\alpha^*\left(t,x\right)\right)$ has linear growth thanks to Assumption \ref{control_drift_ass}. As a consequence, the \textit{closed loop}
 equation
\begin{equation} \label{closed-loop-SDE}
d\bar{X}_t = b^*\left(t,\bar{X}_t\right)dt + \sigma\left(t\right)dW_t,
\end{equation}
\noindent admits a unique strong solution $\bar{X}^{x}$ starting at time $0$ with value $x$, for each $x \in \R^d$, see Theorem 6 in \cite{Veretennikov1982}.
\smallbreak
\begin{prop} \label{verif}
Suppose 
the validity of Assumptions \ref{ass_g},
\ref{control_drift_ass}, \ref{control_costs_ass}.
Let $v$ be the value function defined in \eqref{eq:prob}
supposed to be of class $\shc^{0,1}$ such that
$\left(t,x\right) \mapsto H\left(t,x,\nabla_x v\left(t,x\right)\right)$ has polynomial growth.

\noindent
 Then, the Borel function $\alpha^*$ introduced in Assumption \ref{markov_optimal_control} defines an optimal feedback function for the considered control problem in the sense that for each $x \in \R^d$,
\begin{equation} \label{optimal_alpha}
v\left(0,x\right) = J\left(0,x,\alpha^*\left(\cdot,\bar{X}^x\right)\right).
\end{equation}
\end{prop} 

\begin{prooff} \ (of Proposition \ref{verif}).  Let $x \in \R^d$.

\begin{enumerate} 
\item 
 $v$ is a continuous viscosity solution with polynomial growth 
of \eqref{eq:HJB} and so of \eqref{SLPDE}, with $\left(t,x,y,z\right) \mapsto H\left(t,x,z\right)$ for the non linearity. 
By Remark \ref{RHamilt} we know
that $H^v$ is continuous and by assumption it has
  polynomial growth.
So we apply Lemma \ref{Abstract-Lemma} to deduce that $v$ is of class
 $\shc^{1,2}\left([0,T[,\R^d\right)$ and is a classical solution of \eqref{eq:HJB}.
\item Applying It\^{o}'s formula to $v\left(\cdot, \bar{X}^{x}\right)$ between $0$ and $T_0 \in [0,T[$ and using the fact $v$ is a classical solution of \eqref{eq:HJB} combined with Assumption \ref{markov_optimal_control}, we obtain 
\begin{equation} \label{Eq_T0}
v\left(0,x\right) = v\left(T_0,\bar{X}^{x}_{T_0}\right) + \int^{T_0}_{0}f\left(r,\bar{X}^{x}_r,\alpha^*\left(r,\bar{X}^x_r\right)\right)dr - M_{T_0},
\end{equation}
where 
$$ M_t = \int^{t}_{0} \nabla_x v\left(r,\bar{X}^{x}_r\right)^\top
\sigma\left(r\right) dW_r, \ t \in [0,T[.$$
By the usual BDG (Burkholder-Davies-Gundy) and Jensen's arguments, $ \sup_{t \in [0,T]} \vert  \bar X^x_t \vert$ has all its moments. 
So, \eqref{Eq_T0} implies that the
local martingale $M$ extends continuously  to a true martingale on $[0,T]$
 still denoted by $M$ verifying
$\sup_{t\in[0,T]} \vert M_t\vert \in L^1.$
Indeed $v$ is continuous on $[0,T]\times\R^d$ 
and $v$ (resp. $f$) has polynomial growth in space
(resp. in the second and third variable).
Therefore $M$ is a true martingale.
 Sending $T_0$ to $T$, \eqref{Eq_T0} holds with $T_0$ replaced by $T$ and $v\left(T_0,\bar X^{x}_{T_0}\right)$ replaced by $g\left(\bar{X}^{x}_T\right)$.
Taking the expectation, we obtain
\begin{equation} \label{interm_opt}
v\left(0,x\right) = \mathbb{E}\left(g\left(\bar{X}^{x}_{T}\right) + \int^{T}_{0}f\left(r,\bar{X}^{x}_r,\alpha^*\left(r,\bar{X}^x_r\right)\right)dr \right).
\end{equation}
\item 
The process $\alpha^*_t:= \alpha^*\left(t,\bar{X}^{x}_t\right), \ t \in [0,T]$, 
 belongs to the set $\sha_0$ of admissible controls
and $X = \bar{X}^{x}$, is a solution of \eqref{controlled_SDE}.
Invoking pathwise uniqueness for \eqref{controlled_SDE}, we obtain
 $X^{0,x,\alpha^*}$ coincides with $\bar{X}^x$.
 Then, \eqref{interm_opt} implies \eqref{optimal_alpha}.
\end{enumerate}
\smallbreak
\end{prooff}
\noindent We formulate now a corollary in which is given a representation formula for the value function $v$ involving the optimal feedback function $\alpha^*$.
\begin{corro} \label{HJB_corro}
Let $\nu$ be a Gaussian probability measure
fulfilling Assumption \ref{ass_nu} with associated functions
$m^\nu$ and $Q^\nu$.
We suppose the validity of Assumptions \ref{ass_g}, \ref{control_drift_ass}, 
\ref{control_costs_ass}.
Among the functions fulfilling Assumptions \ref{optimal_control} and \ref{markov_optimal_control},
the value function $v$ is the unique one which is solution of
\begin{equation} \label{RepFormulaHJB}
\left\{
\begin{array}{lll}
 \displaystyle \xi_t = \xi_0 - \int^{t}_{0}\widetilde{b}\left(T-s,\xi_s\right) + b_c\left(T-s,\xi_s,m^\nu\left(T-s\right),Q^\nu\left(T-s\right)\right)ds + \int^{t}_{0}\sigma\left(T-s\right)d\beta_s, \\
\xi_0 \sim \nu, \\
\displaystyle v\left(t,\widehat{\xi}_t\right) = \mathbb{E}\left(\int^{T}_{t}f\left(s,\widehat{\xi}_s,\alpha^*\left(s,\widehat{\xi}_s\right)\right) - \left<\widetilde{b}\left(s,\widehat{\xi}_s\right) - b^*\left(s,\widehat{\xi}_s\right),\nabla_x v\left(s,\widehat{\xi}_s\right)\right>ds + g\left(\widehat{\xi}_T\right)\middle|\widehat{\xi}_t\right),
\end{array}
\right.
\end{equation}
\noindent for all $t \in [0,T]$.
\end{corro}
\begin{proof} 
\noindent The result is a direct consequence of Corollary \ref{RepFormulaControl}, replacing the function $H$ by its expression given in Assumption \ref{markov_optimal_control}.
\end{proof}



\section{A heuristic algorithm}

\label{S5}

\setcounter{equation}{0}

\smallbreak

In this section, we propose a heuristic algorithm to solve the
control problem described in Section \ref{control_problem_section}.
In what follows, the \textit{terminal cost} function $g$ is supposed to
belong to $\shc^1\left(\R^d\right)$.

\noindent
Consider a  regular time grid with time step $\delta t := \frac{T}{n}$ and grid instants $t_k=k\delta t$ for any $k \in [\![0,n]\!]$. 
For $k=n-1,n-2,\cdots ,0$, select arbitrarily $\bar m_{k+1}\,, c_{k+1} \in \R^d$ and $\bar Q_{k+1}\in S^+_d(\R), a_{k+1}\in M_d(\R)$ such that  $
Q_k(t_k) := e^{-a_{k+1}\delta t} \bar{Q}_{k+1} e^{-a^\top_{k+1}\delta t} - 
{\displaystyle \int^{t_{k+1}}_{t_k}e^{-a_{k+1}\left(s-t_k\right)}\Sigma\left(s\right)e^{-a^\top_{k+1}\left(s-t_k\right)}ds}\in S^+_d(\R)\,.$
By Corollary~\ref{HJB_corro}, applied substituting $[0,T]$ with
$[t_k,t_{k+1}]$,
the solution of~\eqref{eq:HJB}
on $[t_k,t_{k+1}]$, with terminal condition $v(t_{k+1},\cdot)$, can be represented for $t\in [t_k,t_{k+1}]$ by
\begin{equation}
\label{eq:discrete}
\left\{
\begin{array}{lll}
\bar \xi_{k+1}&\sim &\mathcal{N}(\bar m_{k+1},\bar Q_{k+1})\\
 Y_{k+1}&=&v(t_{k+1},\bar\xi_{k+1})\\
  m_k(t) &=& e^{-a_{k+1}\left(t_{k+1}-t\right)}\bar m_{k+1} - \left({\displaystyle \int^{t_{k+1}}_{t} e^{-a_{k+1}(s-t)}ds}\right) c_{k+1}
  \\
Q_k(t) &=& e^{-a_{k+1}\left(t_{k+1}-t\right)} \bar{Q}_{k+1} e^{-a^\top_{k+1}\left(t_{k+1}-t\right)} - 
{\displaystyle \int^{t_{k+1}}_{t}e^{-a_{k+1}\left(s-t\right)}\Sigma\left(s\right)e^{-a^\top_{k+1}\left(s-t\right)}ds}\\
\xi_{k,T-t}&=&\bar \xi_{k+1}-{\displaystyle \int_{t_{n-(k+1)}}^{T-t} \big (a_{k+1}\xi_{k,s}+c_{k+1}+b_c(T-s,\xi_{k,s},m_k(T-s),Q_k(T-s))\big )\,ds} \\
&&+{\displaystyle \int_{t_{n-(k+1)}}^{T-t}\sigma (T-s)d\beta_s}\\
\hat \xi_{k,t}&=&\xi_{k,T-t}\\
v(t,\hat \xi_{k,t})&=&{\displaystyle \E\left (\int_{t}^{t_{k+1}} F_k\big (s,\hat \xi_{k,s},\nabla_xv(s,\hat \xi_{k,s})\big) ds+ Y_{{k+1}}\middle|\hat\xi_{k,t}\right )}\ .
\end{array}
\right .
\end{equation}
In the above recursion, $\beta$ denotes a $d$-dimensional Brownian motion on $[0,T]$; for any $k\in [\![0,n-1]\!]$, $(\xi_{k,t})_t$ is a $d$-dimensional process defined on $[t_{n-(k+1)}, t_{n-k}]$ while $(\hat \xi_{k,t})_t$ denotes the 
associated time reversal defined on $[t_k,t_{k+1}]$; the driver $F_k$ defined on $[t_{k},t_{k+1}]\times \R^d\times \R^d$ is such that, 
\begin{equation}
\label{eq:Fk}
F_k(t,x,\delta):=H (t,x,\delta) -\langle a_{k+1}x+c_{k+1},\delta\rangle
=
\min_{a\in A} \left \{f(t, x, a)+ \langle b(t, x, a), \delta\rangle \right \}
-\langle a_{k+1}x+c_{k+1},\delta\rangle\,.
\end{equation}
The idea now is to apply a classical numerical method based on linear regressions to approximate the solution to~\eqref{eq:discrete} recursively in time from $k=n-1$ to $k=0$. 
For each time instant $k$, select arbitrarily $\bar m_{k+1}\,, c_{k+1} \in \R^d$ and $\bar Q_{k+1}\in S^+_d(\R), a_{k+1}\in M_d(\R)$ such that  
\begin{equation}
\label{eq:Q}
Q_k = e^{-a_{k+1}\delta t} \bar{Q}_{k+1} e^{-a^\top_{k+1}\delta t} - \Sigma (t_{k+1})\delta t\in S^+_d(\R)\,.
\end{equation}
Then we propose to approximate $v(t_k,\cdot)$ by $v_k$ obtained by an explicit time discretization scheme of~\eqref{eq:discrete} with time step $\delta t=\frac{T}{n}$ as follows.
\begin{equation}
\label{eq:discreteb}
\left\{
\begin{array}{lll}
\bar \xi_{k+1}&\sim &\mathcal{N}(\bar m_{k+1},\bar Q_{k+1})\\
 Y_{k+1}&=&v_{k+1}(\bar\xi_{k+1})\\
 \xi_{k}&=&\bar \xi_{k+1}-{\displaystyle  \big (a_{k+1}\bar \xi_{k+1}+c_{k+1}+b_c(t_{k+1},\bar \xi_{k+1},\bar m_{k+1},\bar Q_{k+1})\big )\,\delta t} +{\displaystyle \sigma (t_{k+1})\sqrt{\delta t}\varepsilon_k}\\
v_k( \xi_{k})&=&{\displaystyle \E\left (F_k\big (t_{k+1},\bar \xi_{k+1},\nabla_xv_{k+1}(\bar\xi_{k+1})\big) \delta t+ Y_{{k+1}}\middle | \xi_{k}\right )}\ ,
\end{array}
\right .
\end{equation}
where $(\varepsilon_k)_{0\leq k \leq n-1}$ are i.i.d. $d$-dimensional standard Gaussian variables. 
As in the classical literature,
see e.g. \cite{GobetWarin},
we propose to approximate the conditional expectation appearing in~\eqref{eq:discreteb} using Monte-Carlo least squares regression based
on a grid constituted by $N$ independent simulations  $(\xi^i_{k},\bar \xi^i_{k+1})_{1\leq i\leq N}$ for $k\in [\![0,n-1]\!]$.
In that literature,
one generally simulates forwardly that grid.

\noindent
The interest of such \textit{fully backward} representations~\eqref{eq:discrete}-\eqref{eq:discreteb}, where the grid $(\xi^i_{k},\bar \xi^i_{k+1})_{1\leq i\leq N}$  is defined backwardly in time, (like the value function), is twofold.
\begin{itemize}
\item In terms of computer memory: at each time instant $k+1$, the values of the grid are generated on the fly,  $(\xi^i_{k},\bar \xi^i_{k+1})_{1\leq i\leq N}$. Contrary to the standard approach, there is no need to store the whole grid over the whole set of grid instants $k\in [\![ 0,n-1]\!]$. 
\item In terms of the relevance of the grid: at each grid instant, $k+1$ the information acquired on the value function $v(t_{k+1},\cdot)$ and optimal control
  strategy $\alpha^*(t_{k+1},\cdot)$ can be used to adaptively optimize the grid parameter $(a_{k+1},c_{k+1},\bar m_{k+1}, \bar Q_{k+1})$ in order to explore relevant regions of the state space.
\end{itemize}
 We develop some arguments to justify the relevance mentioned above.
 Indeed, as already announced, the target idea is to generate the grid used for regression computations according to the optimally controlled process dynamics. If this were possible, the sensitivity of the driver $F_k$ w.r.t.
 the third variable $\nabla_x v$ would vanish.
 In fact the driver sensitivity w.r.t. $\nabla_x v$ is known to be one
 major cause of the propagation of numerical errors in approximation
 schemes, see e.g.~\cite{gobet16}.
Replacing $\nabla_xv_{k+1}(\bar\xi_{k+1})$ 
by a perturbation $ \nabla_xv_{k+1}(\bar\xi_{k+1})+h$ in the
last equation of \eqref{eq:discreteb}
we obtain 
$$v^h_k(\xi_k) :={\displaystyle \E\Big (F_k\big (t_{k+1},
   \bar \xi_{k+1},\nabla_xv_{k+1}(\bar\xi_{k+1})+h\big) \delta t+ Y_{{k+1}}\,\vert \,
   \xi_{k}\Big )}. $$
The impact on $v_k(\xi_k)$
 can crudely be evaluated by computing the error 
$\E[\vert v^h_k( \xi_k)-v_k(\xi_k)\vert^2].$
Supposing  that no perturbation is impacting $Y_{k+1}$,
fact which will be heuristically justified  in Remark \ref{rem:step5} 1.,
 we have
\begin{eqnarray*}
\E(\vert v^h_k( \xi_k)-v_k(\xi_k)\vert^2)&\le&
\E\Big (\big \vert F_k\big (t_{k+1},\bar \xi_{k+1},\nabla_x v_{k+1}(\bar \xi_{k+1})\big) -F_k\big (t_{k+1},\bar \xi_{k+1},\nabla_x v_{k+1}(\bar \xi_{k+1})+h\big) \,\big \vert^2\Big ).
\end{eqnarray*}

\noindent Suppose from now on the existence of a Borel function
$(t,x,\delta) \mapsto a^*(t,x,\delta),$ such that
\begin{equation} \label{a*}
H(t,x,\delta):= \left \{f(t,x, a^*(t,x,\delta))+ \left<b(t,x, 
a^*(t,x,\delta)),\delta\right> \right \}, \ \left(t,x,\delta\right) \in [0,T]\times \R^d \times \R^d.
\end{equation}
In this case one has  
$ \alpha^*(t,x) = a^*(t,x,\nabla_x v(t,x)), \ (t,x) \in  [0,T]\times \R^d \times \R^d, $ 
where $\alpha^*$ was defined in Assumption \ref{markov_optimal_control}.
Coming back to \eqref{eq:Fk} we get
\begin{equation}
\label{eq:FkBis}
F_k(t,x,\delta):=H (t,x,\delta) -\langle a_{k+1}x+c_{k+1},\delta\rangle
=
 \left \{f(t, x, a^*(t,x,\delta))+ \langle b(t, x, a^*(t,x,\delta)),
 \delta\rangle \right \}
-\langle a_{k+1}x+c_{k+1},\delta\rangle.
\end{equation}

\noindent
A suitable  application
 of the envelope theorem gives
\begin{equation}
\label{eq:partialF}
\frac{\partial F_k}{\partial \delta}(t,x,\delta)= b(t,x,a^*(t,x,\delta))-(a_{k+1}x+c_{k+1}) \,,
\end{equation}
which yields
\begin{eqnarray}
\label{eq:error}
\E\left(\vert v^h_k( \xi_k)-v_k(\xi_k)\vert^2\right)&\le&
\E\Big \vert \langle  
 \int_0^1 
 \frac{\partial F_k}{\partial \delta}(t_{k+1},\bar \xi_{k+1},\nabla_x v_{k+1}(\bar \xi_{k+1})+\theta h) d\theta\,,\, h\rangle \Big \vert^2\nonumber \\
&=&
\E\Big \vert \langle \int_0^1 
 b\big (t_{k+1},\bar \xi_{k+1},a^*(t_{k+1},\bar \xi_{k+1},\nabla_x v_{k+1}(\bar \xi_{k+1})+\theta h)\big ) d\theta -(a_{k+1}\bar \xi_{k+1}+c_{k+1})
\,,\, h
\rangle\Big \vert^2
\nonumber \\
&\leq& \vert h\vert^2 \E\Big \vert \int_0^1  b\big (t_{k+1},\bar \xi_{k+1},a^*(t_{k+1},\bar \xi_{k+1},\nabla_x v_{k+1}(\bar \xi_{k+1})+\theta h)\big ) d\theta -(a_{k+1}\bar \xi_{k+1}+c_{k+1})\Big \vert^2\, .\nonumber
\end{eqnarray}
The above relation highlights the fact that the original idea consisting in 
generating the grid according to a dynamics approaching the optimally controlled process dynamics reduces
the propagation of the error induced by the Monte-Carlo regression 
scheme in terms of least square criteria.

\begin{rem}\label{RFBSDEs}
The above relation also shows that  
 previous idea can be read in the more general perspective of 
the probabilistic representation of a  solution $v$ to a semilinear PDE
  of the type
\eqref{eq:PDE_Intro},
 via an FBSDE. In that general context,  one expects the selected
 drift of the
forward process in the FBSDE 
 to reduce the impact
of the sensitivity of the FBSDE driver 
 with respect to $\nabla_x v$.
\end{rem}




\noindent
Based on that observation, we propose a heuristic algorithm where parameters $(a_{k+1},c_{k+1})$ are adaptively chosen as 
\begin{equation}
\label{eq:ac}
(a_{k+1},c_{k+1})\in \textrm{arg}\min_{a,c}\E\Big \vert b\big (t_{k+1},\bar\xi_{k+1},a^*(t_{k+1},\bar \xi_{k+1},\nabla_xv_{k+1}(\bar\xi_{k+1})\big )-(a \bar\xi_{k+1}+c)\Big \vert^2\,.
\end{equation}
%
\begin{algorithm}
\caption{Fully Backward Monte-Carlo Regression scheme}
\textbf{Initialization} 
Set $v_n=g$; $k=n-1$; select arbitrarily $(\bar m_n,\bar Q_n)\in \R^d\times S^+_d\left(\R\right)$; generate $( \xi^i_n)_{1\leq i\leq N}$ i.i.d. $\sim\,\mathcal{N}(\bar m_n,\bar Q_n)$; set $Y^i_n=g( \xi^i_n)$, for all  $i \in [\![1,N]\!]$. \\
\textbf{while $k\geq 0$ do} 
\begin{enumerate}
\item ${\alpha^i}_{k+1} ={\displaystyle  \underset{a \in A}{\arg\min} \left \{f\left(t_{k+1}, \xi^i_{k+1}, a\right)+ \left<b\left(t_{k}, \xi^i_{k+1}, a\right), \nabla _x  v_{k+1}\left( \xi^i_{k+1}\right)\right>\right \}}$, for all  $i \in [\![1,N]\!].$
\item $\left({a}_{k+1},{c}_{k+1}\right) = \underset{\left(a,c\right) \in M_d\left(\R\right)\times \R^d}{\arg\min} \frac{1}{N}\sum_{i=1}^N\left|a\xi^i_{k+1} + c - b\left(t_{k+1},\xi^i_{k+1},{\alpha}^i_{k+1}\right)\right|^2.$
\item $\bar{m}_{k} = e^{-{a}_{k+1}\delta t} \bar{m}_{k+1} - c_{k+1}\delta t.$
\item ${Q}_{k} = e^{-{a}_{k+1}\delta t}\bar{Q}_{k+1}e^{-{a}^\top_{k+1}\delta t} - \Sigma\left(t_{k+1}\right)\delta t.$
\begin{itemize} 
\item \textbf{If} $\bf{Q_k\in S_d^+\left(\R\right)}$\textbf{:} set $\bar Q_k=Q_k$,  
\item \textbf{Else\hspace{1.1cm}}\textbf{:} set $\bar Q_k=Proj_{S^+_d\left(\R\right)}(Q_k)$; recompute $\bar Q_{k+1}=e^{{a}_{k+1}\delta t}\big (\bar Q_k+ \Sigma(t_{k+1})\delta t\big) e^{{a}^\top_{k+1}\delta t}$; regenerate $( \xi^i_{k+1})_{1\leq i\leq N}\ \textrm{i.i.d.}\ \sim\,\mathcal{N}(\bar m_{k+1},\bar Q_{k+1})$; set $Y^i_{k+1}=v_{k+1}( \xi^i_{k+1})$, for all  $i \in [\![1,N]\!].$
\end{itemize}
\item  
Set ${e}^i_{k+1} = {a}_{k+1}{\xi}^i_{k+1} + {c}_{k+1} - b\left(t_{k+1}, \xi^i_{k+1}, {\alpha}^i_{k+1}\right)$, for all  $i \in [\![1,N]\!].$
\item ${ \xi}^i_{k} = {\xi}^i_{k+1} - \left({a}_{k+1}{\xi}^i_{k+1} + {c}_{k+1} + b_c\left(t_{k+1},{\xi}^i_{k+1},\bar{m}_{k+1},\bar{Q}_{k+1}\right)\right) \delta t + \sigma\left(t_{k+1}\right)\varepsilon^i_{k}\sqrt{\delta t}$, for all  $i \in [\![1,N]\!]$
\item ${v}_k = \underset{P \in P_p\left(\R^d\right)}{\arg\min}\frac{1}{N}\sum_{i=1}^N \left|{Y}^i_{k+1}+ \left(f\left(t_{k+1}, \xi^i_{k+1}, {\alpha}^i_{k+1}\right) - \left<{e}^i_{k+1}, \nabla_x  v_{k+1}\left( \xi^i_{k+1}\right)\right>\right) \delta t - P\left({\xi}^i_k\right) \right|^2.$
\item $Y^i_{k}=Y^i_{k+1}+ \left(f\left(t_{k+1}, \xi^i_{k+1}, {\alpha}^i_{k+1}\right) - \left<{e}^i_{k+1}, \nabla_x  v_{k+1}\left( \xi^i_{k+1}\right)\right>\right) \delta t$, for all  $i \in [\![1,N]\!]$
\item $k-1\leftarrow k.$
\end{enumerate}
\textbf{end while}
\label{algo}
\end{algorithm}
\noindent In the above algorithm, the random variables  $(\varepsilon^i_k\,, \ k\in [\![0,n-1]\!]\,, i\in [\![1,N]\!])$ are i.i.d. according to $\mathcal{N}\left(0,I_d\right)$; 
$Proj_{S^+_d\left(\R\right)}: S_d\left(\R\right)\mapsto S^+_d\left(\R\right)$ denotes the Frobenius projection operator on the closed and convex space
 of semidefinite matrices; for each $p \in \N$, $P_p\left(\R^d\right)$ denotes the set of polynomial functions on $\R^d$ with degree $p$.
\begin{rem}
\label{rem:step5}
\begin{enumerate}
  \item Note that in Step 4, as soon as $Q_k\in S^+_d\left(\R\right)$ then $(Y_{k+1}^i)_{1\leq i\leq N}$ results from the update made at previous iteration at Step 8.  That updating rule corresponds to the \textit{multi-step forward dynamic programming} approach~\cite{gobet16} which is well-known for not inducing any additional bias error that would propagate backwardly during iterations. However, when $Q_k\notin S^+_d\left(\R\right)$, in Step 4, then we have to modify $\bar Q_{k+1}$, re-generate new variables $( \xi^i_{k+1})_{1\leq i\leq N}\ \textrm{i.i.d.}\ \sim\,\mathcal{N}(\bar m_{k+1},\bar Q_{k+1})$ and use the update $Y^i_{k+1}=v_{k+1}( \xi^i_{k+1})$ which adds a bias error.  
	Fortunately, in our numerical simulations it appeared easy to chose a first covariance matrix $\bar Q_n$ so that for all $k\in [\![0,n-1]\!]$ we had $Q_k\in S^+_d$. In that situation, the error propagation is only due to the sensitivity of the driver w.r.t. $\nabla_x v$ which is precisely minimized by our heuristics. 
\item The complexity of Algorithm~\ref{algo}, 
is  comparable to the  traditional Monte-Carlo Regression scheme using a \textit{\textbf{forward grid}}. Indeed, Algorithm~\ref{algo} requires an additional linear regression calculation of order $\mathcal{O}(d^2N)$ at Step 2 which is negligible w.r.t. 
the polynomial regression computations at Step 7 (operated by both algorithms) inducing  $\mathcal{O}(d^4N)$ operations in the specific case considered in simulations where the maximum degree of polynomials is $p=2$. When $Q_k\notin S^+_d$, Algorithm~\ref{algo} requires 
in addition, at Step 4, to implement: a Frobenius projection
 $Proj_{S^+_d\left(\R\right)}(Q_k)$ ($\mathcal{O}(d^3)$), $N$ multiplications of matrices $d\times d$ with vectors $d\times 1$ ($\mathcal{O}(d^2N)$); and  $N$ independent generations of $d$-dimensional Gaussian random variables. These additional operations
  induce a complexity of
 $\mathcal{O}(d^{2}N)$ which
does not increase the original 
 $\mathcal{O}(d^4N)$ complexity. 
\item In terms of memory, as already mentioned, we do not have to store
 the whole regression grid on the whole time horizon constituted of $ndN$
 reals  but only to consider $dN$ reals at each instant.
\end{enumerate}
\end{rem}
\begin{rem}
\noindent Suppose that at each time step $k \in [\![0,n-1]\!]$ the matrix $Q_k$ belongs to $S^+_d\left(\R\right)$. Then, Algorithm \ref{algo} is based on the representation formula appearing in Corollary~\ref{HJB_corro}, on the whole time interval $[0,T]$ with piecewise constant coefficients $a,c$ such that $a(t), c(t) = a_{k+1}, c_{k+1}$ for each $t \in ]t_k,t_{k+1}]$, for each $k \in [\![0,n-1]\!]$.
\end{rem}

\section{Stochastic control of thermostatically controlled loads} \label{Sexample}
\setcounter{equation}{0}

\subsection{Model description}
With the massive integration of variable renewable energies (like wind farms
 or solar panels) into power systems, balancing supply and demand in a real time basis requires to develop new leverages. 
A technical solution is to develop load control schemes in order to
 automatically adapt consumption to generation. 
In this section, we propose to apply Algorithm~\ref{algo} in order to control a large heterogeneous population of air-conditioners on a time horizon $[0,T]$
 such that the overall consumption of the population follows a given target profile,  while preserving the rooms temperatures within users comfort bounds. \\
We consider a hierarchical control scheme introduced in~\cite{callaway}, where the population is aggregated into $d$ clusters of $N^i$ homogeneous loads 
(with same air-conditioners and rooms characteristics) 
for i$\in [\![1,d]\!]$. 
For each cluster $i \in [\![1,d]\!]$, a \textit{local controller} decides at each time step to turn ON or OFF optimally some air-conditioners of cluster $i$, in order to satisfy a \textit{prescribed proportion} of devices with status ON in the cluster. The \textit{prescribed proportion} of devices ON in each cluster, at each time step, is computed by a \textit{central controller} controlling the average rooms temperatures in each cluster, $X^i := \frac{1}{N_i}\sum^{N_i}_{j=1}X^{i,j},$ where $X^{i,j}_t$ is the room temperature associated to load  $j \in [\![1,N_i]\!]$ of cluster $i\in [\![1,d]\!]$. $(X^{i,j}_t)_{0\leq t\leq T}$is supposed to follow the usual thermal dynamics (see \cite{grangereau, seguret} and references therein)
\begin{equation}\label{controlStateij}
X^{i,j}_t = x^{i,j}_0 + \int^{t}_{0}\big (-\theta^i(X^{i,j}_s-x^i_{\rm out})-\kappa^iP^i_{\max}\alpha^{i,j}_s\big )ds + \sigma^{i,j} W^{i,j}_t, \quad t \in [0,T],
\end{equation}
where for any $j \in [\![1,N_i]$, $\sigma^{i,j} > 0$, $\big (W^{i,j})$ are independent real Brownian motions representing model errors and temperature fluctuations inside the room due to local behavior (window, door opening etc.); $x^{i,j}_0$ is the initial temperature; $\kappa^i$ is the heat exchange parameter; $x^i_{out}$ denotes the outdoor air temperature; $1/\theta^i>0$ is the thermal time
constant; $P^i_{\max} > 0$ denotes the maximal power consumption; $\alpha^{i,j}_s\in\{0,1\}$ is the status OFF or ON of load $(i,j)$ at time instant $s\in[0,T]$. \\
We are interested in the problem of the \textit{central controller} who considers the aggregated state process 
 $X := (X^i)_{1\leq i\leq d},$ whose dynamics is obtained by averaging dynamics~\eqref{controlStateij} over $j\in [\![1,N_i]\!]$, for any $i\in [\![1,d]\!]$, 
\begin{equation}\label{controlState}
X^{i}_t = x^{i}_0 + \int^{t}_{0}\big (-\theta^i(X^{i}_s-x^i_{\rm out})-\kappa^iP^i_{\max}\alpha^{i}_s\big )ds + \sigma^i W^{i}_t, \quad t \in [0,T],
\end{equation}
where the control process $\big (\alpha_s=( \alpha^i_s)_{1\leq i\leq d}\,,s\in[0,T]\big )$ taking values in $[0,1]$  prescribes  the proportions of devices ON in each cluster; $x^i_0=\frac{1}{N_i}\sum_{j=1}^{N_i}x^{i,j}_0$; $(\sigma^i)^2=\frac{1}{N_i^2}\sum_{j=1}^{N_i}(\sigma^{i,j})^2$; $(W^i)_{1\leq i\leq d}$ is a $d$-dimensional Brownian motion.
The \textit{central controller} problem can be formulated as a specific instantiation of problem~\eqref{controlled_SDE}-\eqref{cost_function} with the following:
\begin{itemize} 
\item the controlled process $X$ driven by a drift coefficient $b:=(b^i)_{1\leq i\leq d}$ defined on $[0,T]\times\R^d\times[0,1]^d$  s.t. for any $i \in [\![1,d]\!]$ 
$b^i(t,x,a)= -\theta^i\left(x^i- x^i_{out}\right)- \kappa^i P^i_{\max} a^i,$ with the notation $a:=(a^i)_{1\leq i\leq d}$ and $x:=(x^i)_{1\leq i\leq d}$; 
\item the terminal cost $g(x) := \frac{1}{d}\sum_{i=1}^d|x^i-\bar x^i|^2$ where
$\bar x \in \R^d$ denotes given \textit{target} values for the final average temperatures of each cluster; 
\item the running cost defined on $[0,T]\times\R^d\times[0,1]^d$,
  $$f(t,x,a):=\lambda\left(\sum^{d}_{i=1}\rho^i  a^i- r_t\right)^2 +\frac{1}{d}\sum^{d}_{i=1}\Big (\gamma^i (\rho^i a^i)^2 + \eta^i (x^i - x^i_{\max})^2_+ + \eta^i (x^i_{\min} - x^i)^2_+\Big ), $$
where $\rho^i := \frac{N^iP^i_{\rm max}}{\sum_{j=1}^d N^jP^j_{\rm max}}$; 
$\sum_{i=1}^d \rho^ia^i$ gives the overall current consumption of the population as a proportion of the maximum consumption $\sum_{j=1}^d N^jP^j_{\rm max}$; 
 $r : [0,T] \mapsto \R^+_*$  denotes the target consumption profile for the overall consumption as a proportion of the maximum consumption $\sum_{j=1}^d N^jP^j_{\rm max}$;
$\lambda >0$ quantifies  the incentive for the overall consumption to track the target consumption profile $r$; $\gamma^i >0$ quantifies the quadratic penalty favoring smooth consumption profiles for cluster $i$; $\eta^i>0$ is a parameter penalizing excursions outside of the comfort interval  $[x^i_{\min}, x^i_{\max}]$ for cluster $i$ average temperature.
\end{itemize}
Note that $b$ verifies Assumption \ref{control_drift_ass}, $f$ verifies
Assumption \ref{control_costs_ass} and $g$ Assumption \ref{ass_g}.

\subsection{Simulation results}
\label{S62}

Consider the \textit{central controller} problem on a time horizon $T=3600s$,
  with a population of heterogeneous air-conditioners composed of $d=1,2,5,10,15,20$ clusters with $N^i=20$ identical loads in each cluster. We specify the chosen parameters. In each case, $\kappa=2.5^\circ$C/J and $\sigma^i=0.1^\circ$C$s^{\frac{1}{2}}$; $x_{\rm out}=27^\circ $C; $\theta^i[s^{-1}]$
 is chosen arbitrarily
 in  $[0.1, 0.97]$; $P^i_{max}[W]$ is chosen arbitrarily in $[0.5, 5]$; $x_0=\bar x[^\circ$C] is chosen arbitrarily in $[16, 27]$; $x_{\rm min}=\bar x-1.5^\circ $C; $x_{\rm max}=\bar x+1.5^\circ$C; $\eta=1(^\circ$C$)^{-2}$; $\lambda=20$; $\gamma^i$ is chosen arbitrarily
 in $[0.5,1.5]$. The target profile, $r$, used in simulations is obtained as the sum of a nominal profile corresponding to the standard (uncontrolled) behavior of air-conditioners and a deviation: $r= r^{\rm nom}+dev$.   
The standard dynamics of an (uncontrolled) air-conditioner is driven by a cycling rule of ON/OFF decisions 
intended to keep the room temperature in $[x^i_{\rm min},x^i_{\rm max}]$. When the air-conditioner is ON, it stays ON at $P^i_{\rm max}$ until the temperature reaches $x^i_{\rm min}$ then it switches OFF  until the temperature reaches $x^i_{\rm max}$. Then, the air-conditioner
turns ON again and begins a new cycle.
%
The nominal profile $r^{\rm nom}$
 has been generated by averaging the consumption of 1000 sets of $d$ clusters of $N^i$ heterogeneous air-conditioners  simulated independently according to~\eqref{controlStateij}, with $(\alpha^{i,j}_t)_{0\leq t\leq T}$ following the cycling rule of ON/OFF decisions and with independent initial conditions for temperature $x^{i,j}_0\sim \mathcal{N}(x_0^i,1)$ and   ON/OFF status $\alpha^{i,j}_0\in\{0,1\}$. 
The deviation profile $dev_t=\frac{20}{100}*\sin(\frac{2\pi t}{T})$ induces a maximal deviation of $20\%$ from the nominal profile and integrates to zero on the time horizon $[0,T]$ so that the target profile corresponds to the same energy consumed on the period $[0,T]$ as  the nominal profile. 

\noindent
The time step is  $\delta t =60s$. We have implemented Algorithm~1 with a \textbf{\textit{backward grid}} initiated with $\mathcal{N}(m_n=\bar x, Q_n=I_d)$. 
For comparison, we have also implemented the standard Monte-Carlo regression scheme using a \textbf{\textit{forward grid}} simulated according to~\eqref{controlState} with a deterministic control $\alpha_s$ approximating the nominal dynamics (according to the ON/OFF cycling rule) described previously. In both cases, we have used second order polynomials ($p=2$) as basis functions for regressions. We have considered $N=10^2,\,10^3,\, 5\times 10^3,\, 10^4,\, 2\times 10^4,\, 5\times 10^4,\, 10^5$  Monte-Carlo paths for the regression grids.
To evaluate the statistical performances of the \textbf{\textit{forward}} and \textbf{\textit{backward grids}}, we have implemented each algorithm independently  $N_{\rm grid} = 100$ 
 times for each value of $N$. For each run, $i=1,\cdots, N_{\rm grid}$, the value functions estimate $(v^i_k)_{0\leq k\leq n}$ 
(and the corresponding gradients) was used to implement the associated strategy $\alpha^i=(\alpha^i_k)_{0\leq k\leq n}$ on  $M=1000$
 i.i.d. simulations of the Brownian motion $W$, $\omega^1,\cdots ,\omega^j,\cdots ,\omega^M$. Then the resulting cost 
$\shj(\alpha^i,\omega^j):=g(X_T^{0,x_0,\alpha^i}(\omega^j))+\int_0^Tf(r,X^{0,x_0,\alpha^i}_r(\omega^j),\alpha_r)dr\ 
$ 
has been computed. The expected cost has been estimated as 
$
\E[\shj(\alpha^{i},\omega^j)] \approx 
\hat J:=\frac{1}{MN_{\rm grid}}\sum_{i=1}^{N_{\rm grid}}\sum_{j=1}^M \shj(\alpha^{i},\omega^j)\,.
$
The variance of $\hat J$ is estimated by $\hat \sigma^2_{\hat J}$ obtained by replacing, expectations and variances by their empirical approximation based on the sample, $\big (\shj(\alpha^i,\omega^j),\,, i\in [\![1,N_{\rm grid}]\!]\ j\in [\![1,M]\!]\big )$, in the  expression
$
 \hat \sigma^2_{\hat J}\approx Var(\hat J)=\frac{1}{MN_{\rm grid}}\E\left [Var \big (\shj(\alpha^i,\omega^j )\,\vert \alpha^i\big )\right ]+\frac{1}{N_{\rm grid}}Var \left (\E\big [ \shj(\alpha^i,\omega^j )\,\vert \alpha^i\big ]\right ),
$
for each $i$ and $j$.
We have reported on Table~\ref{tab:fwd} (resp. Table~\ref{tab:bwd}) 
the empirical mean $\hat J$ and within parenthesis the empirical
standard deviation $\hat\sigma_{\hat J}$ obtained for each considered pair $(d,N)$ for the \textit{\textbf{forward grid}} (resp. \textit{\textbf{backward grid}}).

\noindent One can observe that the \textit{\textbf{backward grid}} performs surprisingly well providing with high precision the lowest expected cost achieved by both methods (or almost) with only $N=5\times 10^3$ paths whatever the dimension $d$ of the control problem. 
This is consistent with our intuition based on the idea that localizing the grid around the optimally controlled process paths would bring efficiency and reduce the impact of dimension. 
The particularity of this problem is that the optimally controlled process is naturally localized in a small region of the state space because, on the one hand a target value, $\bar x$, is prescribed for the terminal temperatures (by the terminal cost) and on the other hand a target profile is assigned for the overall power consumption. The \textit{\textbf{backward grid}} has the advantage of being initiated around the target state and of following dynamics approaching the optimal strategy. This allows to concentrate the \textit{\textbf{backward grid}} in the small region of interest so that restricting the regression basis to polynomials of order $p=2$ seems already enough to obtain reasonable results.
However, one can observe some cases where the \textit{\textbf{forward grid}} 
 (for $N=10^5$ and $d\leq 5$)  has performed slightly better than the \textit{\textbf{backward grid}}. This can be interpreted by the fact that the \textit{\textbf{forward grid}} knows the initial condition $x_0$ while the \textit{\textbf{backward grid}} has no information about it. To further improve the performances Algorithm~\ref{algo}, an idea would be to find a way to exploit that information on the initial condition. This could constitute the subject of future research.  
\begin{table}[ht]
\centering
\begin{tabular}{ccccccc}
$\bf{N}$ & \bf{d=1}& \bf{d=2} & \bf{d=5}&\bf{d=10}&\bf{d=15}&\bf{d=20}\\
\hline
  $\bf{10^2}$  				&	8.68(0.98)&  17.28(1.01) &42.04(1.32) &  34.79(0.66)&   21.27(0.12)&   18.97(0.09)\\
    $\bf{10^3}$ 			& 7.61(6$e^{-4}$) &  8.24(0.07)&14.83(0.64) &  28.14(0.64)  & 37.91(0.60)  & 34.83(0.45)\\
    $\bf{5\times 10^3}$& 7.60(3$e^{-4}$) &  7.78(2$e^{-3}$)& 8.98(0.21) &  19.84(0.52) &  35.31(0.71) &  33.57(0.52)\\
    $\bf{10^4}$					&7.60(3$e^{-4}$) & 7.77(1$e^{-3}$)& 7.69(0.06) &  16.06(0.38) &  32.20(0.63) &  30.66(0.59)\\
    $\bf{2\times 10^4}$ & 7.60(3$e^{-4}$) & 7.77(2$e^{-4}$)& 7.37(0.02) &  13.58(0.40) &  28.97(0.71) &  28.17(0.67)\\
    $\bf{5\times 10^4}$ & 7.60(3$e^{-4}$)  &7.79(2$e^{-4}$)  &  7.28(2$e^{-3}$) &   7.96(0.25) &  26.69(0.65) &  26.21(0.69)\\
    $\bf{10^5}$					&7.61(3$e^{-4}$)		& 7.78(1$e^{-4}$)&7.27(8$e^{-4}$)&   6.12(0.08) &  22.54(0.56) &  23.26(0.59)
\end{tabular}
\caption{Mean, $\hat J$ (standard deviation, $\hat \sigma_{\hat J}$) of the simulated cost with the \textbf{\textit{forward grid}} strategy}
\label{tab:fwd}
\end{table}
\begin{table}[ht]
\centering 
\begin{tabular}{ccccccc}
$\bf{N}$ & \bf{d=1}& \bf{d=2} & \bf{d=5}&\bf{d=10}&\bf{d=15}&\bf{d=20}\\
\hline
  $\bf{10^2}$  						&	7.61(3$e^{-4}$)	&7.78(7$e^{-4}$) &7.41(6$e^{-3}$) 	&7.31(0.12) &28.14(0.18)&   26.01(0.12)\\
    $\bf{10^3}$ 					& 7.61(3$e^{-4}$) 	&7.77(2$e^{-4}$)	&7.39(1$e^{-3}$) &6.18(3$e^{-3}$) &  8.19(6$e^{-3}$)  & 7.87(1$e^{-2}$)  \\
    $\bf{5\times 10^3}$		& 7.61(3$e^{-4}$) 	&7.77(2$e^{-4}$)	&7.38(8$e^{-4}$)	& 6.17(1$e^{-3}$) & 8.15(2$e^{-3}$) &  7.74(3$e^{-3}$) \\
    $\bf{10^4}$						&7.61(3$e^{-4}$) 	&7.77(2$e^{-4}$)	&7.38(5$e^{-4}$)	&6.17(1$e^{-3}$) 	&8.15(2$e^{-3}$) &7.73(3$e^{-3}$)   \\
    $\bf{2\times 10^4}$ 	& 7.61(3$e^{-4}$) 	&7.77(2$e^{-4}$)	&7.38(3$e^{-4}$)	&6.17(8$e^{-4}$) 	&8.15(1$e^{-3}$) &7.73(2$e^{-3}$) \\
    $\bf{5\times 10^4}$ 	& 7.60(3$e^{-4}$)  &7.79(1$e^{-4}$) &7.38(2$e^{-4}$) & 6.16(5$e^{-4}$) 	&8.14(8$e^{-4}$) & 7.72(1$e^{-3}$) \\
    $\bf{10^5}$						&7.61(3$e^{-4}$)		&7.79(1$e^{-4}$) &7.39(2$e^{-4}$) &  6.16(4$e^{-4}$) &8.14(7$e^{-4}$) & 7.72(9$e^{-4}$) 
\end{tabular}
\caption{Mean $\hat J$ (standard deviation, $\hat \sigma_{\hat J}$) of the simulated cost with the \textbf{\textit{backward grid}} strategy}
\label{tab:bwd}
\end{table}

\section{Appendix} \label{appendix}

\subsection{A sufficient condition to obtain an equivalent probability}

\begin{lem} \label{Girsanov_OU}
\noindent We recall that
$\widetilde b$ was defined in \eqref{E42}. Let $W$ be an
$(\shf_t)_{t\in[0,T]}$-Brownian motion and $X$ be a solution of 
\begin{equation}\label{OUBis}
X_t = X_0 + \int^{t}_{0}\widetilde{b}\left(s,X_s\right)ds + \int^{t}_{0}\sigma\left(s\right)dW_s, \ t \in [0,T],
\end{equation}
\noindent where $X_0$ is a Gaussian random vector independent of $W$. Set $L_t := \sigma\left(t\right)^{-1}\widetilde{b}\left(t,X_t\right), t \in [0,T]$. Then, the Dol\'eans exponential $ \displaystyle \mathcal{E}\left(-\sum^{d}_{i=1}\int^{\cdot}_{0}L^i_sdW^i_s\right) := \exp\left(-\int^{\cdot}_{0}\sum^{d}_{i=1}L^i_sdW^i_s-\frac{1}{2}\int^{\cdot}_{0}\left|L_s\right|^2ds\right)$ is an $(\shf_t)_{t\in[0,T]}$-martingale.
\end{lem}
\begin{proof}
\noindent Following Corollary 5.14 in \cite{karatshreve}, it is sufficient to find a constant time step subdivision $\left(t_n\right)_{n\in \N}$ of $[0,T]$ such that, for all $n \in \N$, 
\begin{equation*}
\mathbb{E}\left(\exp\left(\frac{1}{2}\int^{t_{n+1}}_{t_{n}}\left|L_s\right|^2ds\right)\right) < \infty.
\end{equation*}
\noindent Combining Jensen's inequality and Fubini's theorem, this is fulfilled
 in particular if for all $n \in \N$,
\begin{equation*}
\frac{1}{\delta}\int^{t_{n+1}}_{t_n}\mathbb{E}\left(\exp\left(\frac{\delta \left|L_s\right|^2}{2}\right)\right)ds < \infty,
\end{equation*}
\noindent where $\delta := t_{n+1}-t_n$. 
Let $s \in [0,T]$. Then,
\begin{equation*}
\left|L_s\right|^2 \leq 2 \delta \left|\left| \sigma^{-1} \right|\right|^2_{\infty}\left(\left|\left|a\right|\right|^2_{\infty}\left| X_{s}\right|^2 + \left|\left| c \right|\right|^2_{\infty}\right),\ \mathbb{P}-\rm{a.s},
\end{equation*}
\noindent since $a,c$ are bounded and $\sigma^{-1}$ is
also bounded being continuous on $[0,T]$. Furthermore, 
by item 1. of Lemma \ref{HP_lemma} and \eqref{ENormXi}, $X$ is a Gaussian process 
with mean function $m^X$ (resp. covariance function $Q^X$) solving the first line of equation \eqref{ODE_m} (resp. \eqref{ODE_Q}) with initial
 condition $\mathbb{E}\left(X_0\right)$ (resp. $\mathrm{Cov}\left(X_0\right)$). 
\smallbreak
\noindent Taking into account the fact that $m^X$ is bounded (since continuous), it suffices to find a subdivision such that
\begin{equation*}\label{expmoments}
\mathbb{E}\left(\exp\left(\frac{1}{2}K\delta\left|Z\right|^2\right)\right) < \infty,
\end{equation*}
\noindent where $Z \sim \mathcal{N}\left(0,I_d\right)$ and $K := 4\left|\left| \sigma^{-1} \right|\right|^2_{\infty}\left|\left|a\right|\right|^2_{\infty}\left|\left|Q^X\right|\right|_{\infty}> 0$. This is the case in particular if $K\delta < 1$, which ends the proof.
\end{proof}

\subsection{Proof of the local Lipschitz property of the cost functional $J$}
\begin{lem}\label{loc_Lip_J}
  \noindent Suppose the validity of Assumption \ref{control_drift_ass}.
  Suppose in addition that the functions $g$ and $x \mapsto f\left(t,x,\alpha\right), \left(t,\alpha\right) \in [0,T]\times\sha_0$ are locally Lipschitz with polynomial growth gradient (uniformly in $t$ and $\alpha$). Then, for each $\left(t,\alpha\right)\in[0,T]\times\sha_0$,
\begin{equation*}
x \mapsto J\left(t,x,\alpha\right)
\end{equation*} 
\noindent is locally Lipschitz, uniformly in $t$ and $\alpha$.
\end{lem}
\begin{proof}
\noindent We give here a proof of the local Lipschitz property for the term involving the function $g$ since the other term can be treated in the same way.
\smallbreak
\noindent Let $\left(t,\alpha\right) \in [0,T]\times\sha_0$ and $x,y$ in a compact set of $\R^d$. Let $K$ be the Lipschitz constant of $b$. Using in particular the Cauchy-Schwarz inequality, we get
\begin{align} \label{interm-ineg}
\left|\mathbb{E}\left(g\left(X^{t,x,\alpha}_T\right)\right) - \mathbb{E}\left(g\left(X^{t,y,\alpha}_T\right)\right)\right| &{} \leq \int^{1}_{0}\mathbb{E}\left(\left|\nabla_xg\left(aX^{t,x,\alpha}_T + \left(1-a\right)X^{t,y,\alpha}_T\right)\right|\left|X^{t,x,\alpha}_T - X^{t,y,\alpha}_T\right|\right)da \nonumber \\
&{} \leq e^{KT}\int^{1}_{0}\mathbb{E}\left(\left|\nabla_xg\left(aX^{t,x,\alpha}_T + \left(1-a\right)X^{t,y,\alpha}_T\right)\right|\right)da\left|x-y\right|
\end{align}
\noindent where we have used the estimate $\left|X^{t,x,\alpha}_T - X^{t,y,\alpha}_T\right| \leq e^{KT}\left|x-y\right|$, following from the identity
\begin{equation*}
\left|X^{t,x,\alpha}_r - X^{t,y,\alpha}_r\right| \leq \left|x-y\right| + K\int^{r}_{t}\left|X^{t,x,\alpha}_s - X^{t,y,\alpha}_s\right|ds, \ r \in [t,T],
\end{equation*}
\noindent together with Gronwall's lemma. In view of \eqref{interm-ineg}, the point is proved if $$\int^{1}_{0}\mathbb{E}\left(\left|\nabla_xg\left(aX^{t,x,\alpha}_T + \left(1-a\right)X^{t,y,\alpha}_T\right)\right|\right)da$$ is bounded uniformly in $t,x,y,\alpha$. This follows from polynomial growth of $\nabla_x g$,
classical moment estimates for $\sup_{s\in[t,T]}\left|X^{t,z,\alpha}_s\right|, \ z \in \R^d$ (see for example Corollary 2.5.12 in \cite{krylov}) and the fact $x,y$ lie in a compact set. 
\end{proof}

\subsection{A simplified version of the envelope theorem}

\begin{lem} \label{env_thm}
\noindent Let $\Lambda$ be an arbitrary set and $O$ be an open subset of $\R^d$. Let $x \in \R^d$. Let $F : O\times\Lambda \mapsto \R$ such that for all $\lambda \in \Lambda$, $F\left(\cdot,\lambda\right)$ and $V : x \mapsto \sup_{\lambda \in \Lambda} F\left(x,\lambda\right)$ are differentiable at the point $x$. Suppose also that $\Lambda^*\left(x\right) = \left\{\lambda \in \Lambda, V\left(x\right) = F\left(x,\lambda\right)\right\}$ is not empty. Then, 
\begin{equation*}
\nabla_x V\left(x\right) = \nabla_x F\left(x,\lambda^*_x\right),
\end{equation*}
\noindent for every $\lambda^*_x \in \Lambda^*\left(x\right)$.
\end{lem}
\begin{proof}
\noindent Let $x$ as in the proposition statement and $h \in \R^d$. Let $\lambda^*_x \in \Lambda^*\left(x\right)$. Then, using in particular the differentiability of $F\left(\cdot,\lambda^*_x\right)$ at the point $x$, we get 
\begin{align} \label{ineg_1}
V\left(x+h\right) - V\left(x\right) &{}\geq F\left(x+h,\lambda^*_x\right) - F\left(x,\lambda^*_x\right) \nonumber \\
&{}= \left<\nabla_x F\left(x,\lambda^*_x\right),h\right> + o_0(\left|h\right|). 
\end{align}
\noindent By the differentiability of $V$ at the point $x$, \eqref{ineg_1} implies 
\begin{equation} \label{o_ineg_1}
\left<\nabla_x V\left(x\right) - \nabla_x F\left(x,\lambda^*_x\right),h\right> \geq o_0\left(\left|h\right|\right).
\end{equation}
\noindent Setting $h$ to $-h$ in \eqref{ineg_1} and proceeding as before, we obtain 
\begin{equation} \label{o_ineg_2}
\left<\nabla_x V\left(x\right) - \nabla_x F\left(x,\lambda^*_x\right),h\right> \leq o_0\left(\left|h\right|\right).
\end{equation}
\noindent Combining \eqref{o_ineg_1} and \eqref{o_ineg_2}, we get 
\begin{equation*}
\left<\nabla_x V\left(x\right) - \nabla_x F\left(x,\lambda^*_x\right),\frac{h}{\left|h\right|}\right> \underset{h\to 0}{\longrightarrow} 0,
\end{equation*}
\noindent which forces $\nabla_x V\left(x\right) = \nabla_x F\left(x,\lambda^*_x\right)$. This ends the proof. 
\smallbreak
\end{proof}

\section*{Acknowledgments}

The work was supported by a public grant as part of the
{\it Investissement d'avenir project, reference ANR-11-LABX-0056-LMH,
  LabEx LMH,}
in a joint call with Gaspard Monge Program for optimization, operations research and their interactions with data sciences.

\bibliographystyle{plain}
\bibliography{../../../../BIBLIO_FILE/ThesisLucas}





\end{document}